\documentclass{amsart}
\usepackage{amsmath}
\usepackage{graphicx} 
\usepackage{amsfonts} 
\usepackage{amssymb}
\usepackage{color}
\usepackage{mathrsfs}
\usepackage{epsfig}
\usepackage{url}
\usepackage{mathrsfs,a4wide}
\usepackage{calligra}
\usepackage{mathtools}
\usepackage{esint}
\usepackage{pgf,tikz,pgfplots} 
\theoremstyle{plain}
\newtheorem{theorem}{Theorem}[section]

\newtheorem{lemma}[theorem]{Lemma}

\newtheorem{definition}[theorem]{Definition}
\theoremstyle{definition}
\newtheorem{remark}[theorem]{Remark}
\theoremstyle{remark}
\numberwithin{equation}{section}

\newcommand{\sym}{\mathrm{sym}}
\newcommand{\Tl}{\mathbb{T}}
\newcommand{\T}{\mathcal{T}}
\newcommand{\Ss}{\mathbb{S}}

\newcommand{\ep}{\varepsilon}
\newcommand{\e}{\varepsilon}

\newcommand{\ffi}{\varphi}

\newcommand{\Span}{\mathrm{span}\,}

\newcommand{\C}{\mathbb{C}}
\newcommand{\R}{\mathbb{R}}
\newcommand{\N}{\mathbb{N}}
\newcommand{\Z}{\mathbb{Z}}

\newcommand{\E}{\mathcal{E}}
\newcommand{\F}{\mathcal{F}}
\newcommand{\AS}{\mathcal{AS}}

\newcommand{\di}{\textrm{dist}}

\newcommand{\far}{\mathrm{far}}
\newcommand{\insieme}{{\Omega_{\T_\ep}}}

\newcommand{\ud}{\mathrm{d}}

\newcommand{\Om}{\Omega}
\newcommand{\supp}{\mathrm{supp}\,}
\newcommand{\M}{\mathcal{M}}

\newcommand{\Huno}{\mathcal H^1}

\newcommand{\weakly}{\rightharpoonup}           
\newcommand{\weakstar}{\stackrel{*}{\weakly}}   

\newcommand{\Cu}{\mathrm{Curl}\;}
\newcommand{\SO}{\mathrm{SO}(2)}

\newcommand{\conv}{\mathrm{conv}}
\newcommand{\mdi}{X_\ep^{\gamma,\mathrm{discr}}(\Omega)}
\newcommand{\sdi}{\mathcal{S}_\ep^{\mathrm{discr}}}
\newcommand{\asdi}{\mathcal{AS}_{\ep}^{\gamma,\mathrm{discr}}}

\newcommand{\cut}{\mathsf{C}}
\newcommand{\betalin}{\beta^{\mathrm{cf}}}

\def\XXint#1#2#3{{\setbox0=\hbox{$#1{#2#3}{\int}$}
     \vcenter{\hbox{$#2#3$}}\kern-.5\wd0}}

\newcommand{\res}{\mathop{\hbox{\vrule height 7pt width .5pt depth 0pt
\vrule height .5pt width 6pt depth 0pt}}\nolimits}
\usepackage{latexsym}
\newcommand{\newatop}{\genfrac{}{}{0pt}{1}} 

\makeatletter
\def\@splitop#1#2\@nil{$\mathscr{#1}\!\!$\calligra#2\,\,}
\newcommand*\DeclareCursiveOperator[2]{%

 \newcommand#1{\mathop{\mbox{\@splitop#2\@nil}}\nolimits}}
\makeatother
\DeclareCursiveOperator{\Anew}{A}
\DeclareCursiveOperator{\Bnew}{B}
\DeclareCursiveOperator{\Cnew}{C}
\DeclareCursiveOperator{\Dnew}{D}
\DeclareCursiveOperator{\Enew}{E}
\DeclareCursiveOperator{\Qnew}{Q}

\usepackage{hyperref}

\title[Nonlinear energy for edge dislocations] {$\Gamma$-convergence analysis of the nonlinear self-energy induced by edge dislocations in semi-discrete and discrete models in two dimensions}

\author[R. Alicandro]
{R. Alicandro}
\address[Roberto Alicandro]{DIEI, Universit\`a di Cassino e del Lazio meridionale, via Di Biasio 43, 03043 Cassino (FR), Italy}
\email[R. Alicandro]{alicandr@unicas.it}
\author[L. De Luca]
{L. De Luca}
\address[Lucia De Luca (Corresponding author)]{Istituto per le Applicazioni del Calcolo ``M. Picone'', IAC-CNR, via dei Taurini 19, 00185 Roma, Italy}
\email[L. De Luca]{lucia.deluca@cnr.it}
\author[M. Palombaro]
{M. Palombaro}
\address[Mariapia Palombaro]{DISIM, Universit\`a dell'Aquila, Via Vetoio, 67100 L'Aquila, Italy}
\email[M. Palombaro]{mariapia.palombaro@univaq.it}
\author[M. Ponsiglione]
{M. Ponsiglione}
\address[Marcello Ponsiglione]{Dip. di Matematica, Univ. Roma-I ``La Sapienza'', Piazzale Aldo Moro 5, 00185 Roma, Italy}
\email[M. Ponsiglione]{ponsigli@mat.uniroma1.it}

\begin{document}

\begin{abstract}

We propose nonlinear semi-discrete and discrete models for the elastic energy induced by a finite system of edge dislocations in two dimensions.
Within the {\it dilute regime}, we analyze the asymptotic behavior of the nonlinear elastic energy, as the {\it core-radius} (in the semi-discrete model) and the lattice spacing (in the purely discrete one) vanish. 
Our analysis passes through a linearization procedure  within the rigorous framework of $\Gamma$-convergence.

\end{abstract}
\maketitle

\tableofcontents
\section*{Introduction}
We introduce  some  variational nonlinear models for the elastic energy induced by a finite system of edge dislocations in two dimensions, proposing and analyzing  a purely discrete model settled in the regular triangular lattice,  together with (and relying on)  a variant of the  core radius approach in \cite{SZ}.

We start by describing our variational discrete model in absence of defects. 
We assume to have a finite portion of the regular triangular lattice, with lattice spacing $\ep$ (eventually vanishing), occupying a reference configuration $\Omega\subset\R^2$\,. In this framework, deformations are described through 
discrete strain fields $\beta$ which are defined on the nearest neighbor bonds of the lattice. In the defect-free case, these strains are in fact discrete gradients of some implicitly defined (by integration) deformation function defined on the nodes.
The underlying compatibility condition is that the discrete circulation of $\beta$ around each triangular cell is trivial.
In the present formulation 
the discrete energy is thus a function of the deformation strain $\beta$  and  is given by the sum of two contributions: a two-body interaction potential accounting for the elongation of the bonds, and a three-body interaction term penalizing changes of area of the triangular cells. In the thermodynamic limit (as the lattice spacing vanishes), this model gives back  nowadays classical continuous nonlinear models, where the energy density behaves like the squared distance of the strain from the set of rotations; in the small strain limit, one gets, by linearization, the classical isotropic (continuous) linearized  elasticity (we refer to \cite{BSV,Schm09, ALP} for discrete-to-continuum and linearization results in the context of short and long-range interaction energies). 
Remarkably,  by tuning the  pre-factors  in front of the two contributions of the nonlinear discrete energy, the 
simultaneous discrete-to-continuum limit  and linearization process yield in the limit functional all possible Lam\'e coefficients (see formula \eqref{tuttilame}). 

In this framework, edge dislocations can be introduced as topological singularities of the discrete strain field $\beta$\,, namely, on each cell $\ep T$\,, we enforce the circulation of $\beta$ to belong to (a rotation of) the $\ep$-spaced triangular lattice. 
If non-zero, such a vector represents the so-called {\it Burgers vector} $\ep\xi$\,, which detects and quantifies the presence of an edge dislocation in the triangle $\ep T$\,. We thus identify such a dislocation with the weighted Dirac mass $\ep\xi\delta_{x_{\ep T}}$\,, where $x_{\ep T}$ denotes the barycenter of the triangle $\ep T$\,. 
 A finite distribution of edge dislocations, centered at the points $x^n$ and having Burgers vectors $\ep\xi^n\in\R^2$\,, can be then identified 
with the empirical measure  $\mu=\ep\sum_{n}\xi^n\delta_{x^n}$. 
For any given configuration of dislocations $\mu$, the class of admissible strains is given by those discrete strain fields whose discrete circulation around each triangle $\ep T$ is given by $\mu(\ep T)$.
The presence of edge dislocations enforces some distortion on any admissible strain: since  pure rotations become incompatible,  a resulting stored elastic energy is induced. 
For a single dislocation, such an energy is of order $\ep^2 |\log\ep|$. 
Our analysis is performed in a dilute regime of edge dislocations, where the total stored energy is bounded by $C\ep^2|\log\ep|$\,, which is consistent with the presence of a finite number of dislocations. 
This energy regime has been widely analyzed in terms of $\Gamma$-convergence, 
both for continuous models, also referred to as {\it semi-discrete} models, (see \cite{GLP,DGP} for the linear case and \cite{SZ} for a nonlinear one) and for 
discrete linear energies \cite{ADLPP}. (For analogous results in the linear scalar framework of screw dislocations and of vortices in superconductivity we refer the interested reader to \cite{SS2,P,AC,ACP,ADGP,DL}).
 \par
As the lattice spacing vanishes, the discrepancy between discrete and continuous models should also vanish;  similarly, since we are dealing with a vanishing energy regime of order $\ep^2 |\log\ep|$, also the energy gap between nonlinear and linearized models should vanish.  The purpose of this paper is to rigorously prove these facts, by proposing a purely discrete nonlinear model (the first to our knowledge) and
 building a bridge from discrete to continuous  and from  nonlinear to linear edge dislocation models.
For this purpose, we first  introduce and analyze  a slight variant of the (nonlinear) continuous model studied in \cite{SZ}, which will be instrumental  to the discrete problem. 
 Then we introduce our purely discrete nonlinear model,  inspired from the linear discrete models in \cite{ADLPP, GT,AO}, and we show that it behaves like the continuous one,  as $\ep \to 0$. 
 \par
The nonlinear continuous model we consider is based on the so-called {\it core-radius approach}:
Given a configuration of dislocations $\mu=\ep\sum_{n}\xi^n\delta_{x^n}$, we drill an $\ep$-disc around each dislocation $x^n$\,, and set $\Omega_\ep(\mu):= \Omega \setminus \bigcup_{n} \overline{B}_{\ep}(x^n)$\,.
Here, $\ep$ is (proportional to) the lattice spacing and the discs $B_\ep(x^n)$ represent the so-called {\it cores} of the dislocations $x^n$\,, where plastic effects take place. Removing the cores means to neglect such plastic effects, whose energy contribution is expected to be a lower order perturbation with respect to the elastic energy computed outside the cores, namely, in $\Omega_\ep(\mu)$\,. 
 The admissible strains, in analogy with the discrete setting, are those fields $\beta$ defined on $\Omega_\ep(\mu)$ whose circulation around each $B_\ep(x^n)$ is equal to $\ep\xi^n$. The corresponding elastic energy of $\beta$ (and, in turn, induced by $\mu$) behaves like the integral on $\Omega_\ep(\mu)$ of the squared distance of $\beta$ from the set of rotations $\SO$\,.
  \par 
 Now we describe the further kinematic condition that we impose on the admissible strains. To this end,
assume for the time being that the distribution of dislocations is given by a single dislocation $\ep \xi^0\delta_{x^0}$\,, with  $x^0\in\Omega$ and $\xi^0$ belonging to the $R^0$-rotation of the regular triangular lattice. In such a case, the rigidity estimate \cite{FJM} together with the energetic bound, guarantees that $\beta$ is close to a constant rotation $R\in\SO$ plus an $\ep$-multiple of a linear strain $\beta^{\mathrm{lin}}$ satisfying the circulation condition $\int_{\partial B_\ep(x^0)}\beta^{\mathrm{lin}}t\,\ud\Huno=\xi^0$\,.
In general, the rotation $R$ around which the linearization takes place and the rotation $R^0$ associated to the Burgers vector are decoupled.
However, if $R\neq R^0$ the Burgers vector is in general not consistent with the underlying rigidly rotated lattice, providing then a physically unacceptable configuration.
This suggests to incorporate in the model some extra compatibility condition between the Burgers vectors and the admissible strains in order to prevent such unphysical configurations.
In this respect which options are best suited 
 is in fact still questionable, and a canonical choice seems to be missing in the present literature. 
The main reason is that, in contrast with purely linear models, here we are trying to combine the linear circulation condition, which is built upon the additive decomposition of the linear strain, 
together with the nonlinear energetic framework. 
In this respect, it would be preferable to adopt a multiplicative decomposition of the strain in a plastic, and an elastic part \cite{RC, RSC}, but at the present it is not clear to us how to implement it in a purely discrete framework.   
Moreover, while in the limit as $\ep\to 0$ we obtain a fixed rotation, for positive $\ep$ nonlinear deformations make unavailable a clear notion of local orientation for the deformed lattice. 
In order to overcome this lack of rigidity, in the present formulation we impose that, in the annulus of radii $\ep$ and $\ep^\gamma$ (with $\gamma\in(0,1)$) around each dislocation point, the average of the admissible strain fields coincides with the rotation associated to the corresponding Burgers vector. 
For positive $\ep$\,, rigidity arguments show that such a consistency between the orientation of the Burgers vector and the underlying lattice holds, not only in average, but almost pointwise.
In fact, such an average condition guarantees, in the compactness result, that the limit constant rotation $R$ of  the admissible fields $\beta_\ep$ is also the limit of each of the rotations associated with each Burgers vector.
Such a condition leads us to 
 work under the {\it well separation} assumption, namely, to restrict the class of admissible empirical measures to that for which the distance of two dislocations is larger than (a multiple of) $\ep^\gamma$\,. 
 \par
Within the framework described above, we prove that a given sequence $\{\mu_\ep\}_\ep$ of admissible configurations of edge dislocations, satisfying the energy bound $C\ep^2|\log\ep|$\,, has (once scaled by $\ep$) uniformly bounded mass; therefore,  up to subsequences, $\frac{\mu_\ep}{\ep}$
 converges (in the weak star topology) to a finite sum $\mu=\sum_{k}\xi^k\delta_{x^k}$ of $\R^2$-weighted Dirac deltas.  Furthermore, a sequence of admissible strains $\{\beta_\ep\}_\ep$ compatible with $\{\mu_\ep\}_\ep$\,, converges, up to subsequences, to a constant rotation $R$ that is compatible with $\mu$\,, i.e., the weights $\xi^k$ lie in the $R$-rotated regular triangular lattice. Eventually, we show that $\beta_\ep\sim R+\ep\sqrt{|\log\ep|}\betalin$\,, for some far-field linear strain $\betalin$\,, which is curl-free. This is the content of Theorem \ref{teo:comp}. The corresponding $\Gamma$-convergence result, Theorem \ref{teo:Gamma-conv}, shows that the $\Gamma$-limit of the energy functional with respect to the convergence described above, is given by the sum of a self-energy plus the linearized (around $R$) elastic energy of the limit field $\betalin$\,. The self-energy for such a nonlinear model takes the form of $\sum_{k}\ffi(R^{\mathrm{T}}\xi^k)$\,, where $\ffi$ is a positively $1$-homogeneous function that is obtained by a cell-formula through a relaxation procedure accounting for the underlying lattice structure.
We highlight that such a density $\ffi$ depends only on the corresponding linearized elasticity tensor and on the lattice structure and hence it coincides with  the one computed in linear models \cite{GLP}.
 Loosely speaking, our result shows that the nonlinear elastic energy outside of the cores in the energy regime $\ep^2|\log\ep|$ can be linearized thus obtaining the same $\Gamma$-limit as in the corresponding linear model.
  \par
 Some comments are in order.
 The proof of our result follows the lines of that in \cite{SZ}, where the authors focus on a finite system of fixed (i.e., independent of $\ep$) edge dislocations, whose Burgers vectors lie in the unrotated lattice.
As mentioned above,  our varying measures satisfy a uniform bound on the total variation; therefore, 
from the point of view of compactness properties, our case does not differ much from the case of a fixed system of singularities.
Nevertheless, the situation slightly changes when looking at the $\Gamma$-limit of the elastic energy. Indeed, whereas for a fixed measure no relaxation takes place in the computation of the self-energy, in our model two or more different singularities for positive $\ep$ may converge to the same limiting singularity, thus lowering the value of the self-energy. This means that our $\Gamma$-limit could be  smaller than that computed in \cite{SZ}.
Furthermore,  we stress that the average condition on the strain fields described above is not present in \cite{SZ}, where a possible discrepancy
between the (limit) rotation around which the linearization is performed and the rotation  (that is actually the identity matrix) of the lattice where the Burgers vectors lie, may occur. 
 However, as explained in \cite{MSZlibro} (see also \cite{CGM}), such a gap is not unphysical if one interprets $\Omega$ as the deformed configuration instead of the reference one. Having in mind such an interpretation, an admissible strain field $\beta$ can be locally understood as  an inverse deformation gradient, so that  its  circulation turns out to be a vector of the reference lattice (without any rotation). 
 In fact, while it is quite accepted that the Burgers circuit should be drawn in the deformed configuration, the correct notion of Burgers vector is rather questionable. The one corresponding to the circulation of the (local) inverse of the deformation gradient on a closed path in the deformed configuration is usually referred to as the {\it true} Burgers vector \cite[formula (1-5)]{HL} and, as already explained, it has the advantage to have a topological nature, being quantized on the reference lattice, without being effected by any sort of small elastic fluctuations.
 On the other hand, the so-called {\it local} Burgers vector, defined as the circulation of the deformation gradient on a closed path in the reference configuration, besides being very popular, has a direct constitutive relation with the deformation field. 
 Our proposal of notion of Burgers vector represents a kind of compromise between the two, being defined on the reference configuration but belonging to a rotation of the lattice.
In fact, such a procedure is borrowed from 
 linear models \cite{AO, CL, GLP, DGP}, where the  
 (local) Burgers vector 
 is (somehow tacitly \cite[page 21]{HL}) assumed to belong to the (unrotated) reference lattice. 
Adopting the notion of true Burgers vector within a purely discrete model seems to be a challenging task; in this respect, it would be interesting (also in the linear framework) to 
introduce discrete models defined directly on the deformed configuration, combining the discrete formalism of \cite{LM} with the analysis in \cite{CGM}.
 \par
Summarizing, our model is meant to be a first attempt to combine sound and efficient techniques available in the linear framework together with fundamental discrete nonlinear models. 
Specifically, it relies on the presence of a reference configuration, it is implicitly based on an additive decomposition of the strain, although the energetic setting is nonlinear. 
The resulting modeling choice has a flavor of compromise: we have shown that canonical arguments in linear elasticity theories of dislocations can be pushed to and exploited in a nonlinear framework as well, at the price of some extra care needed 
 to relieve the conflict generated by such an unusual combination. 
 
 Our contribution   represents a first attempt to rigorously analyze in terms of $\Gamma$-convergence a purely discrete nonlinear model for edge dislocations in two dimensions. 

\vskip10pt
{\bf Acknowledgments:} R. Alicandro, L.  De Luca, and M. Palombaro are members of the
Gruppo Nazionale per l'Analisi Matematica, la Probabilit\`a e le loro Applicazioni (GNAMPA) of the Istituto Nazionale di Alta Matematica (INdAM).
The authors thank G. Lazzaroni for interesting and fruitful discussions on  discrete nonlinear elasticity models.

\vskip10pt
{\bf Notation:} Let $X$ be an Euclidean space. For every $\rho>0$ and for every $x\in X$\,, $B_\rho(x)$ denotes the open ball centered at $x$ with radius equal to $\rho$ (with respect to the Euclidean metric). Furthermore, for every $x\in X$ and for every $0<r<R$\,, the symbol $A_{r,R}(x)$ denotes the (open) annulus $B_R(x)\setminus\overline{B}_r(x)$\,. In what follows, $\{e_1,e_2\}$ denotes the canonical basis of $\R^2$\,, i.e., $e_1=(1\quad0)^{\mathrm{T}}$ and $e_2=(0\quad 1)^\mathrm{T}$\,.
 Moreover, $\SO$ is the set of rotations of $\R^2$\,, i.e., the set of matrices $R\in\R^{2\times 2}$ such that $R^{\mathrm{T}} R=\mathrm{Id}$ and $\det R=1$\,. Finally, throughout the paper the symbol $C$ denotes a positive constant that may change from line to line. Whenever we want to stress the dependence of $C$ on some parameters $p_1,\ldots,p_J$ (with $J\in\N$) we write $C_{p_1,\ldots,p_J}$\,.
\section{The semi-discrete model}\label{sec:semidi}
In this section we introduce and analyze the nonlinear semi-discrete model for the elastic energy induced by a family of edge dislocations.
\subsection{Description of the problem}
Here we introduce the main notations that will be used throughout the paper.

Let $\Omega\subset\R^2$ be a bounded, open, and simply connected set with Lipschitz continuous boundary.   
Let $S\subset\R^2$ be a finite subset of $\R^2$ and set $\Ss:=\mathrm{span}_\Z S$\,. Here $\Ss$ represents a reference underlying lattice.
We set $\Pi(\Ss):=\bigcup_{R\in \SO} R\Ss$ and $\mathcal I (\Ss):=\{ R\in \SO: R \Ss = \Ss\}$\,. Let $0<\gamma<1$\,.
Denoting by $\M(\R^2;\R^2)$ the class of $\R^2$-valued Radon measures on $\R^2$\,, for every $\ep>0$ (which is a parameter tuning the lattice spacing) we define the class of admissible edge dislocation measures $X_{\ep}^{\gamma}(\Omega)$ as
\begin{equation}\label{mcont}
\begin{aligned}
X^\gamma_{\ep}(\Omega):=&\Big\{\mu\in\M(\R^2;\R^2)\,:\, \mu=\ep\sum_{n=1}^N\xi^n\delta_{x^n}\,,\,N\in\N,\, x^n\in\Omega,\,\xi^n\in \Pi(\Ss)\,,\\
&\phantom{\mu\in\M(\R^2;\R^2)\,:\,} 
\quad |x^{n_1}-x^{n_2}| \ge 4\ep^\gamma \textrm{ whenever }n_1\neq n_2, \\
&\phantom{\mu\in\M(\R^2;\R^2)\,:\,} \quad  \di (x^n,\partial \Om) \ge 2 \e^\gamma  \textrm{ for all } n=1,\ldots, N  \Big\}\,.
\end{aligned}
\end{equation}
Moreover, for every $R\in\SO$\,, we set
\begin{equation*}
\overline X^R(\Omega):=\Big\{\mu\in\M(\R^2;\R^2)\,:\,\mu=\sum_{k=1}^K\xi^k\delta_{x^k}\,,\,K\in\N,\,x^k\in\Omega,\,\xi^k\in R\Ss\Big\}\,.
\end{equation*}
For every $0<\ep<1$ and for every $\mu=\ep\sum_{n=1}^N  R^n b^n \delta_{x^n}\in X_{\ep}^{\gamma} (\Omega)$, with $b^n\in \Ss$ and $R^n\in\SO$, we set $\Omega_\ep(\mu):=\Omega\setminus\bigcup_{n=1}^N\overline{B}_\ep(x^n)$ and  
we define the class of {\it admissible strains} associated to $\mu$ as
\begin{equation}\label{adm}
\begin{aligned}
\AS^{\gamma}_\ep(\mu):=\Big\{&\beta\in L^2(\Omega;\R^{2\times 2})\,:\,\beta\equiv 0\textrm{ in }\bigcup_{x\in\supp\mu}B_\ep(x)\,,\\
&\phantom{\beta\in L^2(\Omega;\R^{2\times 2})\,:\,} \Cu\beta=0\textrm{ in }\Omega_\ep(\mu), \quad \fint_{A_{\ep,\ep^{\gamma}}(x^n)} \beta \, \ud x \in  R^n\mathcal I(\mathbb S),   \\
&\phantom{\beta\in L^2(\Omega;\R^{2\times 2})\,:\,}
\int_{\partial B_\ep(x^n)}\beta  t\,\ud\Huno= \ep R^n b^n\, \textrm{for every } 
n=1,\ldots,N 
\Big\}\,,
\end{aligned}
\end{equation}
where, here and below, the operator $\Cu\!$ acts row-wise and is understood in the distributional sense, $t$ denotes the tangent vector field to $\partial B_\ep(x^n)$ and the integrand $\beta t$ is intended in
the sense of traces.
Notice that the weight of any admissible singularity cannot be uniquely decomposed into a product $R b$ with $R\in \SO$ and $b\in \Ss$ (as an example, if $\Ss=\Z^2$ then $e_1$ can be written also as the clockwise $\frac\pi 2$-rotation of $e_2$); this is why the condition on the average of $\beta$  in \eqref{adm} is written in terms of an inclusion (into $R^n\mathcal I(\mathbb S)$) rather than an identity.  

We now introduce the nonlinear energy density.
Let $W\in C^0(\R^{2\times 2}; [0,+\infty))$ satisfy the following assumptions:
\begin{align}\label{iprop}\tag{{i}}
&\textrm{$W(\mathrm{Id})=0$\,;}\\ \label{iiprop}\tag{{ii}}
&\textrm{$W(RM)=W(M)$ for every $M\in\R^{2\times 2}, \, R\in\SO$\,; }\\  \label{iiiprop}\tag{{iii}}
&\textrm{$W(MR)=W(M)$ for every $M\in\R^{2\times 2}, \, R\in \mathcal I (\mathbb S)$\,;}\\ \label{ivprop}\tag{{iv}}
&\textrm{there exists a constant $c>0$ such that:
$W(M)\ge c\,\di^2(M,\SO)$ for every $M\in\R^{2\times 2}$}\,;\\ \label{vprop}\tag{{v}}
&\textrm{$W\in C^2$ in a neighborhood of $\SO$}.
\end{align}
Condition \eqref{iprop} states that the reference configuration is stress-free and condition \eqref{iiprop} is the so-called {\it frame indifference}. 
Condition \eqref{iiiprop} says that the energy density $W$ is invariant with respect to the symmetry group of rotations of the underlying lattice (while, if assumed for all rotations, would be an isotropy condition).
Condition \eqref{ivprop} is a coercivity assumption that guarantees compactness properties whereas condition \eqref{vprop} serves to linearize the energy around the equilibria.
 We highlight that, combining conditions \eqref{iprop}, \eqref{iiprop} and \eqref{ivprop}, the well of $W$ is $\SO$\,.
 
In view of property \eqref{vprop}, we can define the $2\times 2\times 2\times 2$ linearized elasticity tensor associated to $W$ as
\begin{equation}\label{tlin}
\C:=\frac{\partial^2 W}{\partial M^2}(\mathrm{Id})\,.
\end{equation}
For every $0<\ep<1$ we define the energy functional $\F^{\gamma}_\ep:\M(\R^2;\R^2)\times L^2(\Omega;\R^{2\times 2})\to [0,+\infty]$ as
\begin{equation}\label{enercont}
\F^{\gamma}_\ep(\mu,\beta):=\left\{\begin{array}{ll}
\int_{\Omega_\ep(\mu)}W(\beta)\,\ud x&\textrm{if }\mu\in X^\gamma_\ep(\Omega)\textrm{ and }\beta\in\AS^{\gamma}_\ep(\mu)\\
+\infty&\textrm{otherwise.}
\end{array}
\right.
\end{equation}
\subsection{Compactness}
Here we state and prove the compactness result for the semi-discrete energy in \eqref{enercont}.
\begin{theorem}\label{teo:comp}
Let $\{(\mu_\ep;\beta_\ep)\}_\ep\subset \M(\R^2;\R^2)\times L^2(\Omega;\R^{2\times 2})$ be such that 
\begin{equation}\label{enbound}
\sup_{\ep>0}\F^{\gamma}_\ep(\mu_\ep,\beta_\ep)\le C\ep^2|\log\ep|\,,
\end{equation}
for some constant $C>0$\,.
Then, up to a subsequence, there exist $\{R_\ep\}_\ep\subset\SO$\,, a rotation $R\in\SO$\,, a measure $\mu\in \overline X^R(\Omega)$\,, and a field $\betalin\in L^2(\Omega;\R^{2\times 2})$ with $\mathrm{Curl}\,\betalin=0$\,,
 such that:
 \begin{align}\label{zeroprop}\tag{{0}}
&\textrm{$R_\ep\to R$\,;}\\ 
\label{dueprop} \tag{{1}}
&\textrm{$\frac{\mu_\ep}{\ep}\weakstar \mu$\,;}\\
\label{unoprop}\tag{{2}}
&\textrm{$\frac{\beta_\ep-R_\ep}{\ep\sqrt{|\log\ep|}}\weakly \betalin$ in $L^2(\Omega;\R^{2\times 2})$\,.} 
\end{align}
\end{theorem}
\begin{definition}\label{conve}
\rm{
In the following, we write $(\mu_\ep;\beta_\ep;R_\ep)\to (\mu;\betalin;R)$ if conditions \eqref{zeroprop}, \eqref{dueprop}, and \eqref{unoprop} of Theorem \ref{teo:comp} are satisfied.
}
\end{definition}
In order to prove Theorem \ref{teo:comp} we follow the lines of the proof of \cite[Proposition 3.5]{SZ}.
The following result provides uniform rigidity estimates {\it \'a la} Friesecke-James-M\"uller \cite{FJM} for a suitable class of domains with holes. It can be proven by arguing verbatim as in \cite[Lemma 3.1]{SZ}.
\begin{lemma}\label{lemma:comeinSZ}
Let $E\subset\R^2$ be a bounded, open set with Lipschitz continuous boundary. Then there exists a constant $C>0$ depending only on $E$ such that the following property holds true. Let $r>0$ and let  $\{x^1,\dots,x^N\}\subset E$ be such that $|x^{n_1}-x^{n_2}|>4r$ for every $n_1,n_2=1,\ldots, N$ with $n_1\neq n_2$ and $\di(x^n,\partial E)>2r$ for every $n=1,\ldots,N$\,. Setting $E_r:=E\setminus \bigcup_{n=1}^{N}\overline B_r(x^n)$\,, for every $u\in H^1(E_r;\R^2)$ there exists $R\in\SO$ such that
\begin{equation*} 
\|\nabla u-R\|_{L^2(E_r;\R^{2\times 2})}\le C\|\di(\nabla u, \SO)\|_{L^2(E_r)}\,.
\end{equation*}
\end{lemma}
\begin{proof}[Proof of Theorem \ref{teo:comp}]
By the energetic bound \eqref{enbound} we have that $\mu_\ep\in X_\ep^\gamma(\Omega)$ and $\beta_\ep\in\AS_\ep^{\gamma}(\mu_\ep)$\,.
For every $0<\ep<1$ let $\mu_\ep:=\ep\sum_{n=1}^{N_\ep}R^n_\ep b_\ep^n\delta_{x_\ep^n}$\,.
We start by proving the compactness property for the measures $\frac{\mu_\ep}{\ep}$\,. 
We set $K_\ep^\gamma:=\lfloor(1-\gamma)\frac{|\log\ep|}{\log 2}-1\rfloor$\,.
Let $n\in\{1,\ldots,N_\ep\}$ be fixed. We have that $A_{\ep,\ep^\gamma}(x_\ep^n)\supset \bigcup_{k=0}^{K_\ep^\gamma}A_{2^k\ep,2^{k+1}\ep}(x_\ep^n)$\,. 
As a consequence of \cite{FJM}, there exists a constant $C>0$ such that for every  $k=0,1,\ldots,K_\ep^\gamma$ 
\begin{equation}\label{rian}
\int_{A_{2^k\ep,2^{k+1}\ep}(x_\ep^n)}\di^2(\beta_\ep,\SO)\ud x\ge C\int_{A_{2^k\ep,2^{k+1}\ep}(x_\ep^n)}|\beta_\ep-R_k|^2\ud x
\end{equation}
for some  rotations $R_k$;  in fact, \eqref{rian} has been proved in \cite{FJM} for gradient fields defined on sets with Lipschitz continuous boundary; moreover, the constant $C$ is invariant with respect to omotheties of the domain. Actually, in our case $\beta_\ep$ is not a gradient, but one can still deduce \eqref{rian} either covering  $A_{2^k\ep,2^{k+1}\ep}(x_\ep^n)$ with two overlapping simply connected sets with Lipschitz continuous boundary, and applying on each of them \cite{FJM}, or introducing a suitable cut in $A_{2^k\ep,2^{k+1}\ep}(x_\ep^n)$ and applying directly \cite{FJM} on the resulting simply connected cut annulus.  
By summing \eqref{rian} over $k$ and over $n$\,, using \eqref{enbound} and property \eqref{ivprop}, we deduce that
\begin{equation}\label{anchedopo}
\begin{aligned}
C\ep^2|\log\ep|\ge&\, C\sum_{n=1}^{N_\ep}\sum_{k=0}^{K_\ep^\gamma}\int_{A_{2^k\ep,2^{k+1}\ep}(x_\ep^n)}|\beta_\ep-R_k|^2\,\ud x\\
\ge&\, C\sum_{n=1}^{N_\ep}\ep^2|b^n_\ep|^2\frac{1}{2\pi} (K_\ep^\gamma+1) \log 2\, 
\ge C\ep^2|\log\ep|(1-\gamma)\Big|\frac{\mu_\ep}{\ep}\Big|(\Omega)\,,
\end{aligned}
\end{equation}
where the last but one inequality follows by applying Jensen inequality and using the very definition of $\AS^\gamma_\ep(\mu_\ep)$\,.
By \eqref{anchedopo}, we deduce that $\big|\frac{\mu_\ep}{\ep}\big|(\Omega)$ is uniformly bounded and hence that there exists a measure $\mu$\,, consisting in a finite sum of vectorial weighted Dirac deltas, such that, up to subsequences, property \eqref{dueprop} holds.

Now, we pass to the proof of property \eqref{unoprop}. To this end,
we define the maps $\bar\beta_\ep:\Omega\to\R^{2\times 2}$ as
\begin{equation*}
\bar\beta_\ep(x):=\ep\chi_{\Omega_\ep(\mu_\ep)}(x)\sum_{n=1}^{N_\ep}(R^n_\ep b^n_\ep)\otimes \mathsf{J}\frac{x-x^n_\ep}{|x-x_\ep^n|^2}\,,
\end{equation*}
where $\mathsf{J}$ is the counter clockwise $\frac\pi 2$-rotation.
One can easily check that
\begin{equation}\label{stimasuibetatilde}
\int_{\Omega_\ep(\mu_\ep)}|\bar\beta_\ep|^2\,\ud x\le C\ep^2N_\ep\sum_{n=1}^{N_\ep}|b^n_\ep|^2|\log\ep|\le C_\gamma\ep^2|\log\ep|\,,
\end{equation}
where the last inequality follows by \eqref{anchedopo}.
By construction, we have that 
$\int_{\partial U}(\beta_\ep-\bar\beta_\ep)t\,\ud\Huno=0$
for every open set $U\subset\Omega$ with $\partial U\subset\Omega_\ep(\mu_\ep)$ smooth;
hence, there exists a map $u_\ep\in H^1(\Omega_\ep(\mu_\ep);\R^2)$ such that $\beta_\ep-\bar\beta_\ep=\nabla u_\ep$ in $\Omega_\ep(\mu_\ep)$\,. By Lemma \ref{lemma:comeinSZ} we thus get that there exist a constant $C>0$ independent of $\ep$ and a sequence of rotations $\{R_\ep\}_\ep\subset\SO$ such that
\begin{equation}\label{stimafond}
\begin{aligned}
\int_{\Omega}|\beta_\ep-R_\ep|^2\ud x\le&
2\int_{\Omega_\ep(\mu_\ep)}|\beta_\ep-\bar\beta_\ep-R_\ep|^2\,\ud x+2\int_{\Omega_\ep(\mu_\ep)}|\bar\beta_\ep|^2\,\ud x+C_\gamma\ep^2\\
\le&2\int_{\Omega_\ep(\mu_\ep)}|\nabla u_\ep-R_\ep|^2\,\ud x+C_\gamma\ep^2|\log\ep|\\
\le&C\int_{\Omega_\ep(\mu_\ep)}\di^2(\nabla u_\ep,\SO)\,\ud x+C_\gamma\ep^2|\log\ep|\\
\le&C\int_{\Omega_\ep(\mu_\ep)}\di^2(\beta_\ep,\SO)\,\ud x+C_\gamma\ep^2|\log\ep|
\le C_\gamma\ep^2|\log\ep|\,,
\end{aligned}
\end{equation}
where we have used also \eqref{anchedopo}, \eqref{stimasuibetatilde}, property \eqref{ivprop}, and \eqref{enbound}\,. By \eqref{stimafond}, up to a subsequence, $\frac{\beta_\ep-R_\ep}{\ep \sqrt{|\log\ep|}} \weakly \betalin$ in $L^2(\Omega;\R^{2\times 2})$ for some  $\betalin\in L^2(\Omega;\R^{2\times 2})$,  so that property \eqref{unoprop} holds true. 

Now we prove that $\Cu\betalin=0$ in the distributional sense. To this end, let $\phi\in C^1_0(\Omega)$ and let $\{\phi_\ep\}_\ep\subset H^1_0(\Omega)$ be a family converging to $\phi$ strongly in $H^1_0(\Omega)$ and such that $\phi_\ep\equiv \phi(x^n_\ep)$ in $B_\ep(x^n_\ep)$ for every $n=1,\ldots,N_\ep$\,. Then, by \eqref{anchedopo}, we get
\begin{equation*}
\begin{aligned}
\langle\Cu\betalin,\phi\rangle=&\lim_{\ep\to 0}\frac{1}{\sqrt{|\log\ep|}}\Big\langle\Cu\frac{\beta_\ep-R_\ep}{\ep},\phi_\ep\Big\rangle=\lim_{\ep\to 0}\frac{1}{\sqrt{|\log\ep|}}\Big\langle\Cu\frac{\beta_\ep}{\ep},\phi_\ep\Big\rangle\\
=&\lim_{\ep\to 0}\frac{\sum_{n=1}^{N_\ep}R^n_\ep b^n_\ep\phi(x^n_\ep)}{\sqrt{|\log\ep|}}=0\,;
\end{aligned}
\end{equation*}
by the arbitrariness of $\phi$ we have that $\Cu\betalin=0$ in the distributional sense.

Furthermore, since $\SO$ is compact, up to a subsequence $R_\ep\to  R$ for some $ R\in \SO$. In order to conclude the  proof of    the theorem  it remains to show that  $\mu\in\overline{X}^R(\Omega)$\,. First, 
for every $n=1,\ldots,N_\ep$\,, let $\tilde R_\ep^n\in\mathcal I(\Ss)$ be such that $\fint_{A_{\ep,\ep^{\gamma}}(x_\ep^n)}\beta_\ep\,\ud x=R_\ep^n\tilde R_\ep^n$\,.
In view of \eqref{anchedopo}, we can assume, up to a subsequence, that $N_\ep\equiv N$ for some $N\in\N$ independent of $\ep$\,. For every fixed $n=1,\ldots,N$\,, it holds
\begin{equation}\label{ultimacosa}
|R_\ep^n\tilde R_\ep^n-R_\ep|^2
=\Big|\fint_{A_{\ep,\ep^{\gamma}}(x^n_\ep)}(\beta_\ep-R_\ep)\,\ud x\Big|^2
\le C_{\gamma}\ep^{2(1-\gamma)}|\log\ep|\,,
\end{equation}
where the last inequality is a consequence of Jensen inequality and of \eqref{stimafond}.
Since $\mathcal{I}(\Ss)$ is compact, we have that, up to a further subsequence, $\tilde R_\ep^n\to \tilde R^n$ for some $\tilde R^n\in\mathcal{I}(\Ss)$\,. This fact together with \eqref{ultimacosa}, yields that $R_\ep^n\to R (\tilde{R}^n)^{\mathrm{T}}$ as $\ep\to 0$ and hence
 $\mu\in\overline{X}^R(\Omega)$\,.
\end{proof}
\begin{remark}
\rm{
In this paper, the assumption that $\Omega$ is simply connected is used to guarantee that the class of $L^2$ curl-free matrix fields on $\Omega$ coincides with the class of gradients of functions in $H^{1}(\Omega;\R^2)$. Such an assumption can be easily dropped if we replace, in the compactness result above and in the 
$\Gamma$-convergence result below, the condition $\betalin \in L^2(\Om;\R^{2\times 2})$ with $\Cu\betalin=0$ by the condition $\betalin=\nabla u$ for some $u\in H^{1} (\Omega;\R^2)$\,.
}
\end{remark}
\begin{remark}\label{nonfunz}
\rm{
As mentioned in the introduction, the condition on the average of the strain field in \eqref{adm} is needed  in Theorem \ref{teo:comp} to guarantee  
the compatibility condition that $R$ is both the asymptotic rotation  around which the linearization is performed and the rotation of the lattice where the weights of the limit singularities lie.
Indeed, dropping such an average condition from the definition of \eqref{adm}, one could consider  $\mu_\ep\equiv \ep b\delta_0$ (with $0\in\Omega$ and $b\in\mathbb{S}$)\,,  $R$ an arbitrary rotation in $\SO$ and the field $\beta_\ep=R+\ep\beta_{\R^2}^{b,\C}$ (where $\beta_{\R^2}^{\zeta,\C}$ is the Green function for the corresponding linearized problem, see \eqref{betapiano} below); such a field (once dropped the average condition of \eqref{adm}) would be an admissible  strain. Now, the weight of the singularity lies on the set $\Ss$, while  $\beta_\ep$ is linearized around the arbitrarily fixed rotation $R$, providing an unphysical lack of rigidity in the class of admissible Burgers vectors. 

Furthermore, we stress that such an average condition should be required on ``thick annuli'' with radii $\ep,\ep^\gamma$ and cannot be replaced by a weaker condition on ``thin'' annuli with radii $\ep,M\ep$ for some $M>1$ (independent of $\ep$)\,,
as shown by the following computation.
Let $\vartheta_\ep:[\ep,1]\to [0,1]$ be such that $\vartheta_\ep(t)\equiv 0$ in $[\ep,M\ep)$\,, $\vartheta_\ep(t)\equiv1$ in $[2M\ep,1]$ and $\vartheta_\ep(t)=\frac{t}{M\ep}-1$ in $[M\ep,2M\ep)$ and let $u_\ep:A_{\ep,1}(0)\to \R^2$ be the function defined by $u_\ep(x):=R(\vartheta_\ep(|x|))x$\,, where
$$
R(\vartheta):=\left(
\begin{array}{ll}
\cos\vartheta&\sin\vartheta
\\
-\sin\vartheta&\cos\vartheta
\end{array}
\right)\,.
$$
By direct computations,
\begin{equation*}
\begin{aligned}
\nabla u_\ep(x)=&\,R(\vartheta_\ep(|x|))
\\
&\,+\frac{\dot{\vartheta}_\ep(|x|)}{|x|}\left(
\begin{array}{ll}
-\sin\vartheta_\ep(|x|)x_1^2+\cos\vartheta_\ep(|x|)x_1x_2&-\sin\vartheta_\ep(|x|)x_1x_2+\cos\vartheta_\ep(|x|)x^2_2
\\
-\cos\vartheta_\ep(|x|)x_1^2-\sin\vartheta_\ep(|x|)x_1x_2&-\sin\vartheta_\ep(|x|)x_2^2-\cos\vartheta_\ep(|x|)x_1x_2
\end{array}
\right)\,,
\end{aligned}
\end{equation*}
whence we deduce that
\begin{equation*}
\int_{A_{\ep,1}(0)}\di^2(\nabla u_\ep,\SO)\,\ud x\le \frac{C}{M^2\ep^2}\int_{A_{M\ep,2M\ep}(0)}|x|^2\,\ud x\le CM^2\ep^2\,.
\end{equation*}
In this example, the average of $\nabla u_\ep$ on the annulus $A_{\ep,M\ep}(0)$ is the identity matrix $R(0)$, while $\nabla u_\ep \to R(1)$  (for instance in $L^2(\Om;\R^{2\times 2})$).
The computation above shows that the fields $\beta_\ep=R(\vartheta_\ep)+\ep\beta_{\R^2}^{b,\C}$ (with $b\in\Ss$), would satisfy the average assumption with $R=R(0)=\mathrm{Id}$ in the annulus $A_{\ep,M\ep}(0)$\,, but converge to the rotation $R(1)$\,.
Therefore, prescribing the condition on the average of the admissible fields  on the annuli of radii $\ep$ and $M\ep$ (for $M$ fixed) does not provide any relationship between the rotation of the lattice where the Burgers vector lies and the rotation around which the energy is linearized.

On the other hand, in order to let our proof (in particular, \eqref{ultimacosa}) work, one can easily see that the average condition in   \eqref{adm}  can be weakened by requiring that 
\begin{equation*}
\di\bigg(\fint_{A_{\ep,\ep^\gamma}(x^n)}\beta\,\ud x,R^n\mathcal{I}(\Ss)\bigg)\le \delta_\ep\,,
\end{equation*}
for some $\delta_\ep>0$ with $\delta_\ep\to 0$ as $\ep\to 0$\,.
}
\end{remark}
\subsection{$\Gamma$-convergence}
This subsection is devoted to the  $\Gamma$-convergence result for the semi-discrete energy in \eqref{enercont}.
In order to define the $\Gamma$-limit we first introduce the self-energy of an edge dislocation. Let $\C$ be a given elasticity tensor in linear elasticity.
For every $\zeta \in\R^2\setminus\{0\}$, the corresponding displacement $u^{\zeta,\C}_{\R^2}$ and strain $\beta_{\R^2}^{\zeta,\C}$\,, induced by the edge dislocation $\zeta\delta_0$ in the whole plane, have been  explicitly computed in the literature (see, for instance, \cite[formula 4.1.25]{BBS}); here, we recall some of their properties we need in our analysis. 
 The strain field $\beta_{\R^2}^{\zeta,\C}$ satisfies the circulation condition
\begin{equation*}
\mathrm{Curl}\,\beta=\zeta\delta_{0}\textrm{ in }\R^2
\end{equation*}
and the equilibrium equation
\begin{equation*}
\mathrm{Div}\,\C\beta=0\textrm{ in }\R^2\,.
\end{equation*} 
In polar coordinates, $\beta_{\R^2}^{\zeta,\C}$ takes the form 
\begin{equation}\label{betapiano}
\beta^{\zeta,\C}_{\R^2}(\rho,\theta):=\frac{1}{\rho} \Gamma^{\zeta,\C}(\theta)\,, 
\end{equation}
where 
(see \cite[Remark 7]{GLP}) $\Gamma^{\zeta,\C}$ is the unique minimizer of $\int_{0}^{2\pi}\C\Gamma:\Gamma\,\ud\theta$ among the functions $\Gamma$ of the form 
$$
\Gamma(\theta):=f(\theta)\otimes(-\sin\theta;\cos\theta)+g\otimes(\cos\theta;\sin\theta)\, ,
$$
with $g\in\R^2$ and $f\in C^0([0,2\pi];\R^2)$ satisfying  $f(0)=f(2\pi)$ and $\int_{0}^{2\pi}f(\omega)\,\ud\omega=\zeta$\,.
The optimal  $g^{\zeta,\C}$ and  $f^{\zeta,\C}$ are uniquely determined by the vector $\zeta$ and the tensor $\C$\,.
\begin{remark}\label{zeromean}
We highlight that
\begin{equation}\label{medianulla}
\int_{0}^{2\pi}\Gamma^{\zeta,\C}(\theta)\,\ud\theta=0\,,
\end{equation}
 and hence, for all $0<r<R$,  
$\int_{A_{r,R}(0)}\beta_{\R^2}^{\zeta,\C}\,\ud x=0$\,.
Indeed, let $\bar{f}^{\zeta,\C}$ be the $2\pi$-periodic extension of $f^{\zeta,\C}$ 
and consider
\begin{equation*}
\hat\Gamma^{\zeta,\C}(\theta):=-\Gamma^{\zeta,\C}(\theta+\pi)=\bar{f}^{\zeta,\C}(\theta+\pi)\otimes(-\sin\theta;\cos\theta)+g^{\zeta,\C}\otimes(\cos\theta;\sin\theta)\,;
\end{equation*} 
Notice that the circulation of $\hat\Gamma^{\zeta,\C}$ coincides with that of $\Gamma^{\zeta,\C}$\,.
Since, by convexity, 
\begin{equation*}
\begin{aligned}
\int_{0}^{2\pi}\C\Big(\frac 1 2\Gamma^{\zeta,\C}+\frac 1 2\hat\Gamma^{\zeta,\C}\Big):\Big(\frac 1 2\Gamma^{\zeta,\C}+\frac 1 2\hat\Gamma^{\zeta,\C}\Big)\,\ud\theta\le& \frac 1 2\int_{0}^{2\pi}\C\Gamma^{\zeta,\C}:\Gamma^{\zeta,\C}\,\ud\theta+\frac 1 2  \int_{0}^{2\pi}\C\hat\Gamma^{\zeta,\C}:\hat\Gamma^{\zeta,\C}\,\ud\theta\\
= \int_{0}^{2\pi}\C\Gamma^{\zeta,\C}:\Gamma^{\zeta,\C}\,\ud\theta\,,
\end{aligned}
\end{equation*}
by strict convexity we deduce that 
\begin{equation}\label{precirem}
\Gamma^{\zeta,\C}(\cdot)=\hat\Gamma^{\zeta,\C}(\cdot)=-\Gamma^{\zeta,\C}(\cdot+\pi)\,;
\end{equation} 
therefore, $\Gamma^{\zeta,\C}$ is an odd function and hence \eqref{medianulla} holds true.
\end{remark}
For every $\zeta\in\R^2\setminus \{0\}$ we set 
\begin{equation}\label{servealimsup}
\psi^{\C}(\zeta):=\int_{0}^{2\pi}\frac{1}{2}\C\,\Gamma^{\zeta,\C}(\theta):\Gamma^{\zeta,\C}(\theta)\,\ud\theta
=\frac{1}{|\log r|}\int_{A_{r,1}(0)}\frac{1}{2}\C\beta^{\zeta,\C}_{\R^2}:\beta^{\zeta,\C}_{\R^2}\,\ud x\,,\qquad  0<r<1\,.
\end{equation}
By \cite[Proposition 1]{DGP} (see also \cite[Corollary 6]{GLP}), we have
\begin{equation}\label{perliminf}
\psi^{\C}(\zeta) = \lim_{\frac{r_2}{r_1} \to +\infty}\psi_{r_1,r_2}^{\C}(\zeta)\,,
\end{equation}
where 
\begin{equation}\label{perliminfgiusta}
\psi_{r_1,r_2}^{\C}(\zeta):=\frac{1}{\log \frac{r_2}{r_1}}\min_{\newatop{\beta\in L^2(A_{r_1,r_2}(0);\R^{2\times 2})}{\newatop{\Cu\beta=0\textrm{ in }A_{r_1,r_2}(0)}{\int_{\partial B_r(0)}\beta t\,\ud\Huno=\zeta}}}\int_{A_{r_1,r_2}(0)}\C\beta:\beta\,\ud x\,.
\end{equation}
Moreover,  for every $\zeta\in\R^2$ and for every $s\in\R$
\begin{equation}\label{bounds}
\psi^\C(s\zeta)=s^2\psi^{\C}(\zeta) \, .
\end{equation}
Finally, for every $b\in\Ss$ we define 
\begin{equation}\label{generffi}
\ffi^{\C}(b):=\min\bigg\{\sum_{i=1}^{I}|z_i|\psi^{\C}(b_i)\,:\, z_i\in\Z\,,\, b_i\in\Ss\,,\,N\in\N\,,\,\sum_{i=1}^{I}z_ib_i=b\bigg\}
\end{equation}
(with $\ffi^\C(b) =0$ for $b=0$).
We are now in a position to state our $\Gamma$-convergence result for the nonlinear energy $\F_{\ep}^\gamma$ defined in \eqref{enercont}.
\begin{theorem}\label{teo:Gamma-conv}
The following $\Gamma$-convergence result holds true.
\begin{itemize}
\item[(i)] ($\Gamma$-liminf inequality) Let $R\in\SO$\,, $\mu=\sum_{k=1}^K\xi^k\delta_{x^k}\in \overline{X}^R(\Omega)$\,, $\betalin\in L^2(\Omega;\R^{2\times 2})$ with $\mathrm{Curl}\,\betalin=0$\,. For every $\{(\mu_\ep;\beta_\ep;R_\ep)\}_\ep\subset\M(\R^2;\R^2)\times L^2(\Omega;\R^{2\times 2})\times \SO$ such that $(\mu_\ep;\beta_\ep;R_\ep)\to (\mu;\betalin;R)$ (in the sense of Definition \ref{conve}), it holds
\begin{equation}\label{liminfformula}
\sum_{k=1}^K\ffi^\C(R^{\mathrm{T}}\xi^k)+\int_{\Omega}\frac 1 2\C R^{\mathrm{T}}\betalin:R^{\mathrm{T}}\betalin\,\ud x\le\liminf_{\ep\to 0}\frac{1}{\ep^2|\log\ep|}\F^{\gamma}_\ep(\mu_\ep,\beta_\ep)\,.
\end{equation}
\item[(ii)] ($\Gamma$-limsup inequality) Let $R\in\SO$\,, $\mu=\sum_{k=1}^K\xi^k\delta_{x^k}\in \overline{X}^R(\Omega)$ and $\betalin\in L^2(\Omega;\R^{2\times 2})$ with $\mathrm{Curl}\,\betalin=0$\,.
 Then, there exists $\{(\mu_\ep;\beta_\ep;R_\ep)\}_\ep\subset\M(\R^2;\R^2)\times L^2(\Omega;\R^{2\times 2})\times \SO$ such that $(\mu_\ep;\beta_\ep;R_\ep)\to (\mu;\betalin;R)$ and
\begin{equation}\label{limsupformula}
\sum_{k=1}^K\ffi^\C(R^{\mathrm{T}}\xi^k)+\int_{\Omega}\frac 1 2\C R^{\mathrm{T}}\betalin:R^{\mathrm{T}}\betalin\,\ud x\ge\limsup_{\ep\to 0}\frac{1}{\ep^2|\log\ep|}\F^{\gamma}_\ep(\mu_\ep,\beta_\ep)\,.
\end{equation}
\end{itemize}
\end{theorem}
Before giving the proof of Theorem \ref{teo:Gamma-conv} we state and prove the following linearization lemma for curl-free matrix fields, 
arguing as in the proof of  \cite[Proposition 3.11]{SZ}. We mention that this kind of linearization results in presence of defects generalizes the case of gradient fields originally treated in \cite{DMNP}, all relying on the Rigidity Estimate in \cite{FJM}.
\begin{lemma}\label{linearlemma}
Let $r>1$\,.
For every $\delta>0$\,, let $\zeta_\delta\in\R^2$ and $\beta_\delta\in L^2( A_{1, r}(0); \R^{2\times 2})$ be such that $\int_{\partial B_1(0)}\beta_\delta t\,\ud \Huno=\zeta_\delta$ and $\Cu\beta_\delta=0$ in $A_{1, r}(0)$\,. Assume that  ${\beta_\delta}\to R$ in $L^2(A_{1,r}(0);\R^{2\times 2})$ and $\frac{\zeta_\delta}{\delta}\to\zeta$ (as $\delta\to 0$), for some $\zeta\in\R^2$ and $R\in\SO$. Then
\begin{equation}\label{nuova}
\liminf_{\delta\to 0}\frac{1}{\delta^2}\int_{A_{1,r}(0)}W(\beta_\delta)\,\ud x\ge \log r\,\psi_{1,r}^{\C}(R^{\mathrm{T}}\zeta)\,.
\end{equation}
\end{lemma}
\begin{proof}
In order to prove \eqref{nuova} we argue by contradiction.
Assume that there exists a vanishing sequence $\{\delta_l\}_{l\in\N}\subset (0,+\infty)$ and a constant $\sigma>0$ such that 
\begin{equation}\label{nuovasba}
\frac{1}{\delta_l^2}\int_{A_{1,r}(0)}W(\beta_{\delta_l})\,\ud x< \log r\,\psi_{1,r}^{\C}(R^{\mathrm{T}}\zeta)-\sigma\,.
\end{equation}
By \eqref{nuovasba}, in view of property \eqref{ivprop} and of the Rigidity Estimate in \cite{FJM} (see \cite[Proposition 3.3]{SZ} for a specific formulation of such a rigidity result suited for curl-free fields on annuli, as well as Lemma \ref{lemma:comeinSZ} above), we get that there exist a constant $C>0$ and a sequence of constant rotations $\{\bar R_{\delta_l}\}_{l\in\N}\subset\SO$ such that
\begin{equation}\label{p10}
\|\beta_{\delta_l}-\bar R_{\delta_l}\|^2_{L^2(A_{1,r}(0);\R^{2\times 2})}\le C\log r\,\psi^\C_{1,r}(R^{\mathrm{T}}\zeta)\delta_l^2\,.
\end{equation}
Clearly, by triangular inequality,
$\bar R_{\delta_l}\to R$ as $l\to +\infty$\,.
By assumption, the field $\frac{\bar R_{\delta_l}^{\mathrm{T}}\beta_{\delta_l}-\mathrm{Id}}{\delta_l}$ 
 is curl-free, and hence, setting $\mathsf{C}:=\{(x_1;0)\,:\, 1\le x_1\le r\}$\,, we have that there exists a sequence $\{v_{\delta_l}\}_{l\in\N}\subset H^1(A_{1,r}(0)\setminus \mathsf{C};\R^2)$ such that $\int_{A_{1,r}(0)}v_{\delta_l}\,\ud x=0$ and $\frac{\bar R_{\delta_l}^{\mathrm{T}}\beta_{\delta_l}-\mathrm{Id}}{\delta_l}=\nabla v_{\delta_l}$ in $A_{1,r}(0)\setminus \mathsf{C}$\,. 
By construction, 
\begin{equation}\label{ilsalto}
[v_{\delta_l}]
=\bar R_{\delta_l}^{\mathrm{T}}\frac{\zeta_{\delta_l}}{\delta_l}\qquad\textrm{on }\mathsf{C}\,,
\end{equation}
where $[v]$ denotes the jump of the function $v$ in the sense of traces.
By \eqref{p10}, we obtain
\begin{equation}\label{p2}
\int_{A_{1,r}(0)}|\nabla v_{\delta_l}|^2\,\ud x\le C\log r\,\psi^\C_{1,r}(R^{\mathrm{T}}\zeta)\,,
\end{equation}
whence we deduce that, up to a subsequence, $ v_{\delta_l}\weakly v$ in $H^1(A_{1,r}(0)\setminus\mathsf{C};\R^2)$ for some $ v\in H^1(A_{1,r}(0)\setminus\mathsf{C};\R^2)$\,; moreover, by \eqref{ilsalto} and recalling that $\bar R_{\delta_l}\to R$\,, we get
\begin{equation}\label{ilsaltofinale}
[v]=R^{\mathrm{T}}\zeta\qquad\textrm{on }\mathsf{C}\,.
\end{equation}
We are now in a position to linearize the energy around the identity. 
To this end, we define the sequence $\{\chi_{\delta_l}\}_{l\in\N}$ of characteristic functions
\begin{equation*}
\chi_{\delta_l}(x):=\left\{
\begin{array}{ll}
1&\textrm{if }|\nabla v_{\delta_l}|<\delta_l^{-\alpha}\\
0&\textrm{otherwise in }A_{1,r}(0)\setminus\mathsf{C}\,,
\end{array}
\right.
\end{equation*}
with $0<\alpha<1$ arbitrarily fixed.
By \eqref{p2} it follows that $\chi_{\delta_l}\to 1$ in measure and hence $\nabla  v_{\delta_l}\chi_{\delta_l}\weakly \nabla v$ in $L^2(A_{1,r}(0)\setminus\mathsf{C};\R^{2\times 2})$\,.
By Taylor expansion,
using \eqref{iiprop},  we have
\begin{equation}\label{zeta0}
\begin{aligned}
\frac{1}{\delta_l^2}\int_{A_{1,r}(0)}W(\beta_{\delta_l})\,\ud x
\ge& \,\frac{1}{\delta_l^2}\int_{A_{1,r}(0)}\chi_{\delta_l}W(\bar{R}_{\delta_l}^{\mathrm{T}}\beta_{\delta_l})\,\ud x\\
=&\,\frac{1}{\delta_l^2}\int_{A_{1,r}(0)\setminus\mathsf{C}}\chi_{\delta_l}W(\mathrm{Id}+\delta_l\nabla v_{\delta_l})\,\ud x\\
\ge&\,\int_{A_{1,r}(0)\setminus\mathsf{C}}\Big(\frac 1 2 \C(\chi_{\delta_l}\nabla v_{\delta_l}):(\chi_{\delta_l} \nabla  v_{\delta_l})
-\frac{\chi_{\delta_l}}{\delta_l^2}\omega(\delta_l|\nabla v_{\delta_l}|)\Big)\,\ud x\,,
\end{aligned}
\end{equation}
with $\frac{\omega(t)}{t^2}\to 0$ as $t\to 0^+$\,. 

Notice that 
\begin{equation*}
\frac{\chi_{\delta_l}}{\delta_l^2}\omega(\delta_l|\nabla v_{\delta_l}|)=|\nabla v_{\delta_l}|^2{\chi_{\delta_l}}\frac{\omega(\delta_l|\nabla v_{\delta_l}|)}{\big(\delta_l|\nabla v_{\delta_l}|\big)^2}\,,
\end{equation*}
which is the product of a uniformly bounded sequence ($\{|\nabla v_{\delta_l}|^2\}_{l\in\N}$) in $L^1(A_{1,r}(0))$ and a uniformly vanishing sequence in $L^\infty(A_{1,r}(0))$\,.  
Therefore, by \eqref{zeta0} and by lower semicontinuity, we conclude
\begin{equation}\label{zeta3}
\liminf_{l\to +\infty} \frac{1}{\delta_l^2}\int_{A_{1,r}(0)}W(\beta_{\delta_l})\,\ud x\ge \int_{A_{1,r}(0)\setminus\mathsf{C}}\frac 1 2 \C\nabla v: \nabla v\,\ud x\ge \psi^{\C}_{1,r}(R^{\mathrm{T}}\zeta)\,,
\end{equation}
where the last inequality follows by the very definition of  $\psi^{\C}_{1,r}$ in \eqref{perliminfgiusta} and by \eqref{ilsaltofinale}.
Since \eqref{zeta3} contradicts \eqref{nuovasba}, we have that \eqref{nuova} holds true.
\end{proof}
With Lemma \ref{linearlemma} in hand, we are in a position to prove the $\Gamma$-$\liminf$ inequality in Theorem \ref{teo:Gamma-conv}.
\begin{proof}[Proof of Theorem \ref{teo:Gamma-conv}(i)]
We can assume without loss of generality that \eqref{enbound} is satisfied.
For every $\ep>0$ let $\mu_\ep=\ep\sum_{n=1}^{N_\ep}\xi^n_\ep\delta_{x^n_\ep}\in X_\ep^\gamma(\Omega)$\,.
In view of \eqref{anchedopo}, we have that $|\frac{\mu_\ep}{\ep}|(\Omega)\le C_\gamma$\,, and hence we may assume that, up to a subsequence, $N_\ep=\hat N$ for some $\hat N$ independent of $\ep$\,. Furthermore, up to passing to a further subsequence, we may assume that each of the points $x^n_\ep$ converges to some point in a finite set $\{y^j\}_{j=1,\ldots,J}\subset\bar\Omega$ (with $K\le J\le \hat N$), where $y^j=x^k\in \Om$ for $j=1,\ldots,K$\,.  
Let $0<\rho<1$ be such that the balls $B_{2\rho}(y^j)$ are pairwise disjoint and $\bigcup_{k=1}^K{B_{2\rho}(x^k)}\subset\Omega$\,. 
Then, setting
\begin{equation*}
\E^{\mathrm{far}}_\ep:=\int_{\Omega\setminus \bigcup_{j=1}^J\overline{B}_{\rho}(y^j)}W(\beta_\ep)\,\ud x\,,\quad 
\E^{k}_\ep:=\int_{\Omega_{\ep}(\mu_\ep)\cap B_{\rho}(x^k)}W(\beta_\ep)\,\ud x \textrm{ (for every $k=1,\ldots,K$)}\,,
\end{equation*}
we have that
\begin{equation}\label{somma}
\F_\ep^{\gamma}(\mu_\ep,\beta_\ep)\ge \E^{\mathrm{far}}_\ep+\sum_{k=1}^{K}\E^{k}_\ep\,.
\end{equation}
By arguing verbatim as in the proof of \cite[Proposition 3.11]{SZ}, namely, linearizing $W$ around the limit rotation $R$ and using the lower semicontinuity of the elastic energy with respect to the weak convergence in $L^2$\,, one can check that
\begin{equation}\label{farfield}
\liminf_{\ep\to 0}\frac{1}{\ep^2|\log\ep|} \E^{\mathrm{far}}_\ep\ge \int_{\Omega\setminus \bigcup_{j=1}^J\overline{B}_{\rho}(y^j)}\C R^{\mathrm{T}}\betalin:R^{\mathrm{T}}\betalin\,\ud x\,.
\end{equation}
Now we claim that, for every $k=1,\ldots,K$
\begin{equation}\label{core}
\liminf_{\ep\to 0}\frac{1}{\ep^2|\log\ep|}\E^{k}_\ep\ge \ffi^\C(R^{\mathrm{T}}\xi^k)\,.
\end{equation}
Notice that \eqref{somma}, \eqref{farfield} and \eqref{core} yield
\begin{equation*}
\liminf_{\ep\to 0}\frac{1}{\ep^2|\log\ep|}\F_\ep^{\gamma}(\beta_\ep)\ge \sum_{k=1}^K \ffi^\C(R^{\mathrm{T}}\xi^k)+ \int_{\Omega\setminus \bigcup_{j=1}^J\overline{B}_{\rho}(y^j)}\C R^{\mathrm{T}}\betalin:R^{\mathrm{T}}\betalin\,\ud x\,,
\end{equation*}
whence \eqref{liminfformula} follows by sending $\rho\to 0$\,.

Therefore, in order to conclude the proof it is enough to show that for any fixed $k=1,\ldots,K$\,, formula \eqref{core} holds true.
To this end, we fix $k\in\{1,\ldots,K\}$ and we set 
\begin{equation*}
\tilde\mu_\ep^k:=\ep\sum_{x^n_\ep\in B_\rho(x^k)}\xi^n_\ep\delta_{x^n_\ep}\,;
\end{equation*}
trivially, 
$
\frac{\tilde\mu_\ep^k}{\ep}\weakstar\xi^k\delta_{x^k}
$
and
\begin{equation}\label{dsum2}
\frac{\tilde\mu_\ep^k(B_\rho(x^k))}{\ep}\to\xi^k\qquad\textrm{as }\ep\to 0\,.
\end{equation}
Up to a subsequence, we may assume that the cardinality  of $\supp\tilde\mu_\ep^k$ is given by some $N^k\in\N$ independent of $\ep$\,.
Up to a relabeling, we can also assume that (recall that $k$ is fixed)
$\supp\tilde\mu_\ep^k=\{x^1_\ep,\ldots, x^{N^k}_\ep\}$\,.
Let (for $\ep$ small enough) $g_\ep^k:\big(\frac{|\log\rho|}{|\log\ep|},1\big)\to \N$ denote the function that associates to any $q\in \big(\frac{|\log\rho|}{|\log\ep|},1\big)$ the number $g_\ep^k(q)$ of connected components of $\bigcup_{n=1}^{N^k}B_{\ep^q}(x^n_\ep)$\,. We have that $g_\ep^k\equiv N^k$ in the interval $[\gamma,1)$ and $g^k_\ep$ is monotonically non decreasing in $\big(\frac{|\log\rho|}{|\log\ep|},1\big)$\,; hence, it can have at most $L^k_\ep\le N^k$ discontinuities.
Up to passing to a subsequence, we can assume that $L^k_\ep=L^k$ for some $L^k$ independent of $\ep$\,.
 Let $\{q_\ep^{k,l}\}_{l=1,\ldots,L^k}\subset\big(\frac{|\log\rho|}{|\log\ep|},\gamma\big)$ denote the set of discontinuity points of the   function $g_\ep^k$ with
$q_\ep^{k,l}<q_\ep^{k,l+1}$ for every $l=1,\ldots,L^k -1$\,.
There exists a finite set $\triangle:=\{q^{k,1},\ldots,q^{k,\tilde L^k}\}\subset(0,\gamma]$ with $q^{k,m}<q^{k,m+1}$ and $\tilde L^k\le L^k$ such that, up to a subsequence, $\{q_\ep^{k,l}\}_{\ep}$ converges to some point in $\triangle$\,, as $\ep\to 0$ for every $l=1,\ldots,L^k$\,. Moreover, we set $q^{k, 0}:=0$ and $q^{k,\tilde L^k+1}:=1$\,. 
Let $\lambda>0$ be such that $2\lambda<\min \{q^{k,m+1}-q^{k,m}:\, m\in\{0, 1, \dots,\tilde L^k\}\}$; then, for $\ep$ small enough, the function $g^k_\ep$ is constant in the interval $[q^{k,m}+\lambda, q^{k,m+1}-\lambda]$\,, its value being denoted by $I_\ep^{k,m}$. 

One can easily see  (in fact, the following arguing is a part of the so called ball construction as done for instance in \cite{Sa}, to which we refer the reader for further details) that for every $m=0,1,\ldots,\tilde L^k$ there exists  a family of 
$I_\ep^{k,m}$
 annuli  
 \begin{equation}\label{anelloni}
 C_\ep^{k,m, i}:=B_{\ep^{q^{k,m}+\lambda}}(z_\ep^{k,m,i})\setminus \overline{B}_{\ep^{q^{k,m+1}-\lambda}}(z_\ep^{k,m,i})
 \end{equation}
with $z_\ep^{k,m,i}\in B_{\rho}(x^k)$ and  $i=1,\ldots,I_\ep^{k,m}$\, such that the following properties hold true:  The annuli $C_\ep^{k,m,i}$ are pairwise disjoint, contained in $B_{2\rho}(x^k)$ and for all $m=0,1,\ldots,\tilde L^k$ 
$$
\bigcup_{n=1}^{N^k} B_\ep(x^n_\ep)\subset\bigcup_{i=1}^{I_\ep^{k,m}}B_{\ep^{q^{k,m+1}-\lambda}}(z_\ep^{k,m,i})\, .
$$
Setting $\xi^{k,m,i}_\ep:=\mu_\ep (B_{\ep^{q^{k,m+1}-\lambda}}(z_\ep^{k,m,i}))$\,,
in view of \eqref{dsum2}, we have that (for every $m=0,1,\ldots,\tilde L^k$)
\begin{equation}\label{dsum4}
\sum_{i=1}^{I_\ep^{k,m}} \frac{\xi_\ep^{k,m,i}}{\ep}\to\xi^k\qquad\textrm{as }\ep\to 0.
\end{equation}
Up to subsequences, we have that (for every $m=0, 1,\ldots, \tilde L^k$)
 $I^{k,m}_\ep=I^{k,m}$ is independent of $\ep$\,; we claim that
 \begin{equation}\label{dsum41}
 \frac{\xi_\ep^{k,m,i}}{\ep}\to \hat{\xi}^{k,m,i}\qquad\textrm{for every }i=1,\ldots, I^{k,m}\,,
 \end{equation}
 for some vectors 
 \begin{equation}\label{limitepesi}
 \hat\xi^{k,m,i}\in R\Ss
 \end{equation}
 with 
 \begin{equation}\label{sommapesi}
 \sum_{i=1}^{I^{k,m}} \hat{\xi}^{k,m,i}=\xi^k\,.
 \end{equation}
Indeed, up to passing to a further subsequence, by construction, 
$\xi_\ep^{k,m,i}=\ep\sum_{n=1}^{N^{k,m,i}} R_\ep^n b_\ep^n$\,, for some $N^{k,m,i}\in\N$\,, $R_\ep^n\in\SO$\,, and $b_\ep^n\in\Ss$\,. 
By \eqref{ultimacosa}, we thus have that, for every $n=1,\ldots,N^{k,m,i}$ there exists a rotation $\tilde R_{\ep}^{n}\in\mathcal{I}({\Ss})$ such that
 \begin{equation*}
 |R_\ep^n b^n_\ep-R_\ep(\tilde R^n_\ep)^{\mathrm{T}} b^n_\ep|^2\le C_{\gamma}\ep^{2(1-\gamma)}\sqrt{|\log\ep|}\,.
 \end{equation*}
 whence, using that $(\tilde R_\ep^n)^{\mathrm{T}} b^n_\ep\in \Ss$ and that $|R_\ep-R|\to 0$\,, we deduce \eqref{limitepesi};  finally, \eqref{sommapesi} follows immediately from \eqref{dsum4} and \eqref{dsum41}.

Fix $r>1$\,.
For every $m=0,1,\ldots, \tilde L^k$ we set
 $H_\ep^{k,m}:=\big\lfloor (q^{k,m+1}-q^{k,m}-2\lambda)\frac{|\log\ep|}{\log r}\big\rfloor$ and, for every $h=1,\ldots, H_\ep^{k,m}$ and for every $i=1,\ldots,  I^{k,m}$\,, we define $A_{\ep}^{k,m,i,h}:=A_{r^{h-1}\ep^{q^{k,m+1}-\lambda}, r^{h}\ep^{q^{k,m+1}-\lambda}}(z_\ep^{k,m,i})$\,.
Recalling the definition of $C_\ep^{k,m,i}$ in \eqref{anelloni} and setting
$$
\hat\beta^{k,m,i,h}_\ep(y):=\beta_\ep(r^{h-1}\ep^{q^{k,m+1}-\lambda}y+z_\ep^{k,m,i})\,,\qquad \textrm{ for every }y\in A_{1,r}(0)\,, 
$$
we get
\begin{equation}\label{p0}
\begin{aligned}
\frac{1}{\ep^2}\int_{C_\ep^{k,m,i}}W(\beta_\ep)\,\ud x\ge&\, \sum_{h=1}^{H^{k,m}_\ep}\frac{1}{\ep^2}\int_{A_{\ep}^{k,m,i,h}}W(\beta_\ep)\,\ud x\\
=&\, \sum_{h=1}^{H^{k,m}_\ep}\frac{1}{(r^{1-h}\ep^{1-q^{k,m+1}+\lambda})^2}\int_{A_{1,r}(0)}W(\hat\beta^{k,m,i,h}_\ep)\,\ud y\,,
\end{aligned}
\end{equation}
where we have used the change of variable $y=r^{1-h}\ep^{-q^{k,m+1}+\lambda}(x-z_\ep^{k,m,i})$\,.

By construction, $\hat\beta^{k,m,i,h}_\ep\in L^2(A_{1,r}(0);\R^{2\times 2})$\,, $\Cu\hat\beta^{k,m,i,h}_\ep=0$ and 
\begin{equation*}
\begin{aligned}
\zeta_\ep^{k,m,i,h}:=\int_{\partial B_1(0)}\hat\beta^{k,m,i,h}_\ep\,t\,\ud\Huno=&\,r^{1-h}\ep^{-q^{k,m+1}+\lambda}\int_{\partial B_{r^{h-1}\ep^{q^{k,m+1}-\lambda}}(z_\ep^{k,m,i})}\beta_\ep\,t\,\ud\Huno\\
=&\,r^{1-h}\ep^{-q^{k,m+1}+\lambda}\xi_\ep^{k,m,i}\,.
\end{aligned}
\end{equation*}
By \eqref{dsum41}, we get that $\frac{\zeta_\ep^{k,m,i,h}}{r^{1-h}\ep^{1-q^{k,m+1}+\lambda}}\to\hat\xi^{k,m,i}$ as $\ep\to 0$\,.
Moreover, by \eqref{stimafond}, we have that
\begin{equation*}
\begin{aligned}
\int_{A_{1,r}(0)}|\hat\beta_\ep^{k,m,i,h}-R_\ep|^2\,\ud y=&\,(r^{1-h}\ep^{-q^{k,m+1}+\lambda})^2\int_{A_\ep^{k,m,i,h}}|\beta_\ep-R_\ep|^{2}\,\ud x\\
\le&\, C_\gamma (r^{1-h}\ep^{1-q^{k,m+1}+\lambda})^2|\log\ep|\to 0\qquad\textrm{as }\ep\to 0\,,
\end{aligned}
\end{equation*}
which implies that $\hat\beta_\ep^{k,m,i,h}\to R$ in $L^2(A_{1,r}(0);\R^{2\times 2})$\,. Therefore, we can apply Lemma \ref{linearlemma} with $\delta=r^{1-h}\ep^{1-q^{k,m+1}+\lambda}$\,, thus obtaining that
\begin{equation}\label{p00}
\frac{1}{(r^{1-h}\ep^{1-q^{k,m+1}+\lambda})^2}\int_{A_{1,r}(0)}W(\hat\beta^{k,m,i,h}_\ep)\,\ud y\ge \log r\,\psi^{\C}_{1,r}\big(R^{\mathrm{T}}\hat \xi^{k,m,i}\big)-\sigma_\ep\,,
\end{equation}
for some family $\{\sigma_\ep\}_\ep$ (independent of $k,m,i$ and $h$) with $\sigma_\ep\to 0$ (as $\ep\to 0$).
By summing \eqref{p00} over $h=1,\ldots,H_\ep^{k,m}$\,,
by \eqref{p0}, we have that
\begin{equation}\label{core0}
\begin{aligned}
\frac{1}{\ep^2}\int_{C_\ep^{k,m,i}}W(\beta_\ep)\,\ud x\ge&\, H_\ep^{k,m}\big(\log r\,\psi^{\C}_{1,r}\big(R^{\mathrm{T}}\hat \xi^{k,m,i}\big)-\sigma_\ep\big)\\
\ge&\, (q^{k,m+1}-q^{k,m}-2\lambda)|\log\ep|\Big(\psi^\C_{1,r}(R^{\mathrm{T}} \hat \xi^{k,m,i})-\frac{\sigma_{\ep}}{\log r}\Big)\\
&\quad-\log r\,\psi_{1,r}^{\C}(R^{\mathrm{T}} \hat \xi^{k,m,i})+\sigma_\ep\,.
\end{aligned}
\end{equation}
Notice that by \eqref{perliminf} and by homogeneity (see \eqref{bounds}), $|\psi_{1,r}^\C(\zeta) - \psi^\C(\zeta)|\le  \omega_r |\zeta|^2$ for some modulus of continuity $\omega_r$ with $\omega_r\to 0 $ as $r\to +\infty$\,.
Now, summing \eqref{core0} first over $i=1,\ldots, I^{k,m}$ and then over $m=0,1,\ldots,\tilde L^k$\,, and using \eqref{limitepesi} and \eqref{sommapesi} together with the very definition of $\ffi^{\C}$ in \eqref{generffi}, we obtain
 \begin{equation*}
 \begin{aligned}
 \liminf_{\ep\to 0}\frac{1}{\ep^2|\log\ep|}\E_\ep^{k}\ge&\liminf_{\ep\to 0}\sum_{m=0}^{\tilde L^k}(q^{k,m+1}-q^{k,m}-2\lambda)
 \sum_{i=1}^{I^{k,m}}\psi^\C_{1,r}(R^{\mathrm T} \hat\xi^{k,m,i})\\
&\,+\liminf_{\ep\to 0}\sum_{m=0}^{\tilde L^k}(q^{k,m+1}-q^{k,m}-2\lambda)I^{k,m}\Big(-\frac{\sigma_\ep}{\log r}\Big)\\
&\,+\liminf_{\ep\to 0}\frac{1}{|\log\ep|}\sum_{m=0}^{\tilde L^k}\sum_{i=1}^{I^{k,m}}\big(-\log r\,\psi_{1,r}^{\C}(R^{\mathrm{T}} \hat \xi^{k,m,i})+\sigma_\ep\big)\\
&\ge \sum_{m=0}^{\tilde L^k}(q^{k,m+1}-q^{k,m}-2\lambda)\sum_{i=1}^{I^{k,m}}\Big(\psi^\C(R^{\mathrm T} \hat\xi^{k,m,i})-\omega_r |\hat\xi^{k,m,i}|^2\Big)\\
 \ge&\big(1-2\lambda(\tilde L^k+1)\big)\ffi^{\C}(R^{\mathrm{T}}\xi^k)-\big(1-2\lambda(\tilde L^k+1)\big)C\omega_r\,,
 \end{aligned}
 \end{equation*}
 whence \eqref{core} follows sending $r\to +\infty$ and $\lambda\to 0$\,, thus concluding the proof of (i).
\end{proof}
We conclude the proof of Theorem \ref{teo:Gamma-conv} by constructing the recovery sequence.
\begin{proof}[Proof of Theorem \ref{teo:Gamma-conv}(ii)]
By standard density arguments in $\Gamma$-convergence it suffices to prove the claim for
$\psi^{\C}(R^{\mathrm{T}}\xi^k)=\ffi^{\C}(R^{\mathrm{T}}\xi^k)$ (for every $k=1,\ldots,K$)
 and $\betalin\in L^\infty(\Omega;\R^{2\times 2})$\,. Since
$\Omega$ is simply connected and 
since $\Cu\betalin=0$\,, there exists a map $u^{\betalin}\in W^{1,\infty}(\Omega;\R^2)$
 such that $\betalin=\nabla u^{\betalin}$\,.

Setting $\cut:=\{(x_1;0)\,:\,x_1\ge 0\}$\,,
for every $\zeta\in\R^2$ let $u^{R,\zeta}\in C^\infty(\R^2\setminus\cut;\R^2)$ be such that
\begin{equation}\label{23marzo}
\nabla u^{R,\zeta}=R\beta_{\R^2}^{\zeta,\C}\,,
\end{equation}
with $\beta_{\R^2}^{\zeta,\C}$ defined in \eqref{betapiano}.
For every $k=1,\ldots,K$\,, setting $\cut^k:=x^k+\cut$\,, we define the map  $u_\ep^{R,k}\in C^\infty(\R^2\setminus\cut^k;\R^2)$
as
\begin{equation}\label{dispfon}
u_\ep^{R,k}(x):=u^{R,R^{\mathrm{T}}\xi^k}(x-x^k)+\sqrt{|\log\ep|}u^{\betalin}(x^k)\,;
\end{equation}
then $\nabla u_\ep^{R,k}(\cdot)=R\beta^{R^{\mathrm{T}}\xi^{k},\C}_{\R^2}(\cdot-x^k)$
and
\begin{equation}\label{salcos}
[u_\ep^{R,k}]=\xi^{k}\textrm{ on }\cut^k\textrm{ (locally in the sense of traces)}.
\end{equation}
We define the map $u_\ep^{\far}\in H^{1}\big(\Omega_{\ep}(\mu)\setminus\bigcup_{k=1}^{K}\cut^k;\R^{2}\big)$ as
\begin{equation}\label{dispfar}
 u_\ep^{\far}(x):=\sum_{k=1}^Ku^{R,R^{\mathrm{T}}\xi^k}(x-x^k)+\sqrt{|\log\ep|}u^{\betalin}(x)\,.
\end{equation}
Let $\phi\in C^1([1,2];[0,1])$ be such $\phi(1)=1$ and $\phi(2)=0$\,; we define the function $u_\ep\in H^{1}\big(\Omega_{\ep}(\mu)\setminus\bigcup_{k=1}^{K}\cut^k;\R^{2}\big)$ as 
\begin{equation}\label{recose}
u_\ep(x):=\left\{
\begin{array}{ll}
u_\ep^{R,k}(x)&\textrm{if }x\in A_{\ep,\ep^{\gamma}}(x^k)\textrm{ for some }k\,,\\
\phi\big(\frac{|x-x^k|}{\ep^{\gamma}}\big)u_\ep^{R,k}(x)+\big(1-\phi\big(\frac{|x-x^k|}{\ep^{\gamma}}\big)\big)u_\ep^{\far}(x)&\textrm{if }x\in A_{\ep^{\gamma},2\ep^{\gamma}}(x^k)\textrm{ for some }k\,,\\
u_\ep^{\far}(x)&\textrm{if }x\in\Omega_{2\ep^{\gamma}}(\mu)\,.
\end{array}
\right.
\end{equation}
Finally, we set $\beta_\ep:=R+\ep\nabla u_\ep$ in $\Omega_\ep(\mu)$ and $\beta_\ep\equiv 0$ elsewhere in $\Omega$\,, so that $\beta_\ep\in L^{2}(\Omega;\R^{2\times 2})$\,.
Notice that, although $u_\ep$ jumps across $\cut^k$\,, we have that $\Cu\,\beta_\ep=0$ in $\Omega_\ep(\mu)$\,,
 since, by \eqref{salcos}, the tangential derivatives of the traces agree along $\cut^k$\,.
Setting $\mu_\ep\equiv\ep\mu$ and $R_\ep\equiv R$
 for every $\ep>0$\,, we claim that $\{(\mu_\ep;\beta_\ep;R_\ep)\}_{\ep}$ is a recovery sequence.
First, in view of Remark \ref{zeromean},
 $\beta_\ep\in\AS_\ep^{\gamma}(\mu_\ep)$\,. Second, by the very definition of $R_\ep$ and $\mu_\ep$\,, properties \eqref{zeroprop} and \eqref{dueprop} are trivially satisfied. We now show that this is the case also for condition \eqref{unoprop}.
To this end, we first notice that, by \eqref{betapiano},
\begin{equation}\label{solufo}
\begin{aligned}
&\|\nabla u^{R,R^{\mathrm{T}}\xi^k}\|_{L^2(\Omega\setminus\overline{B}_\ep(0);\R^{2\times 2})}\le C\sqrt{|\log\ep|}\,,\\
&\frac{1}{\sqrt{|\log\ep|}}\|\nabla u^{R,R^{\mathrm{T}}\xi^k}\|_{L^1(\Omega\setminus\overline{B}_\ep(0);\R^{2\times 2})}\to 0\quad\textrm{for every }k=1,\ldots,K\,,
\end{aligned}
\end{equation}
which implies, in particular,
\begin{equation}\label{peranello00}
\frac{1}{\ep\sqrt{|\log\ep|}}
(\beta_\ep-R_\ep)\chi_{A_{\ep,\ep^{\gamma}}(x^k)}=\frac{\nabla u^{R,R^{\mathrm{T}}\xi^k}(\cdot-x^k)}{\sqrt{|\log\ep|}}\weakly 0\qquad\textrm{in }L^2(\Omega;\R^{2\times 2})\,.
\end{equation}
Moreover, recalling \eqref{dispfon} and \eqref{dispfar}, for every $k=1,\ldots,K$ and for every $x\in A_{\ep^{\gamma},2\ep^{\gamma}}(x^k)$\,, 
we have that
\begin{equation*}
\begin{aligned}
u_\ep(x)=&\,u^{R,R^{\mathrm{T}}\xi^k}(x-x^k)+\Big(1-\phi\Big(\frac{|x-x^k|}{\ep^\gamma}\Big)\Big)\Big(\sum_{j\neq k}u^{R,R^{\mathrm{T}}\xi^j}(x-x^j)+\sqrt{|\log\ep|}u^{\betalin}(x)\Big)\\
&\,+\phi\Big(\frac{|x-x^k|}{\ep^\gamma}\Big)\sqrt{|\log\ep|}u^{\betalin}(x^k)\,;
\end{aligned}
\end{equation*}
hence,
for $\ep$ small enough and for $x\in A_{\ep^{\gamma},2\ep^{\gamma}}(x^k)$\,, we have
\begin{equation}\label{peranello0}
\begin{aligned}
|\nabla u_\ep(x)|\le &\, |\nabla u^{R,R^{\mathrm{T}}\xi^k}(x-x^k)|\\
&\,+\Big|1-\phi\Big(\frac{|x-x^k|}{\ep^\gamma}\Big)\Big|
\Big(\sum_{j\neq k}|\nabla u^{R,R^{\mathrm{T}}\xi^j}(x-x^j)|+\sqrt{|\log\ep|}|\betalin(x)|\Big)\\
&\,+\Big|\nabla\phi\Big(\frac{|x-x^k|}{\ep^\gamma}\Big)\Big|\Big(\sum_{j\neq k}|u^{R,R^{\mathrm{T}}\xi^j}(x-x^j)|+\sqrt{|\log\ep|}|u^{\betalin}(x)-u^{\betalin}(x^k)|\Big)\\
\le&\,C\ep^{-\gamma}+C(1+\sqrt{|\log\ep|})+C\ep^{-\gamma}(1+\ep^\gamma\sqrt{|\log\ep|})
\le C\ep^{-\gamma}\,,
\end{aligned}
\end{equation}
where we have used that $|\nabla u^{R,R^{\mathrm{T}}\xi^j}(x)|\le \frac{C}{|x|}$\,, that $\betalin\in L^\infty(\Omega;\R^{2\times 2})$\,, and that $|u^{\betalin}(x)-u^{\betalin}(x^k)|\le \|\betalin\|_{L^\infty(\Omega;\R^{2\times 2})}|x-x^k|$\,.
By \eqref{peranello0} we have immediately that
\begin{equation}\label{peranello}
\begin{aligned}
\frac{1}{\ep\sqrt{|\log\ep|}}\|\beta_\ep-R_\ep\|_{L^2(A_{\ep^{\gamma},2\ep^{\gamma}}(x^k);\R^{2\times 2})}=&\,\frac{1}{\sqrt{|\log\ep|}}\|\nabla u_\ep\|_{L^2(A_{\ep^{\gamma},2\ep^{\gamma}}(x^k);\R^{2\times 2})}\\
\le&\,\frac{C}{\sqrt{|\log\ep|}}\to 0\qquad\textrm{as }\ep\to 0\,.
\end{aligned}
\end{equation}
Therefore, by \eqref{peranello00} and \eqref{peranello}, using again \eqref{solufo}, we have that
\begin{equation}\label{finale}
\begin{aligned}
\frac{\beta_\ep-R_\ep}{\ep\sqrt{|\log\ep|}}=&-\frac{R_\ep}{\ep\sqrt{|\log\ep|}}\chi_{\Omega\setminus\Omega_\ep(\mu_\ep)}
\\
&\,+\sum_{k=1}^{K}\frac{1}{\ep\sqrt{|\log\ep|}}(\beta_\ep-R_\ep)\chi_{A_{\ep,2\ep^\gamma}(x^k)}\\
&\,+\frac{1}{\sqrt{|\log\ep|}}\nabla u_\ep^{\mathrm{far}}\chi_{\Omega_{2\ep^\gamma}(\mu_\ep)}
\weakly \betalin\qquad\textrm{in }L^2(\Omega;\R^{2\times 2})\,,
\end{aligned}
\end{equation}
i.e., \eqref{unoprop}.

Now we conclude by proving that the sequence $\{(\mu_\ep;\beta_\ep)\}_{\ep}$ satisfies \eqref{limsupformula}.
Since $W\in C^0(\R^{2\times 2})$ and $\|\beta^{R^{\mathrm{T}}\xi^k,\C}_{\R^2}\|_{L^\infty(\R^2\setminus \overline{B}_{\ep}(0);\R^{2\times 2})}\le \frac{C}{\ep}$ for some $C>0$ (independent of $\ep$), for every $\lambda>1$ we have
\begin{equation}\label{dentrissimo}
\frac{1}{\ep^2|\log\ep|}\int_{A_{\ep,\lambda\ep}(x^k)}W(\beta_\ep)\,\ud x
\le \frac{C\lambda^2}{|\log\ep|}\to 0\qquad\textrm{as }\ep\to 0\,.
\end{equation}
Moreover, note that, by assumption \eqref{vprop}, there exist $\rho>0$ and $c_\rho>0$ such that 
\begin{equation}\label{ok}
W(M)\le c_\rho\di^2(M,\SO)\qquad\textrm{ for every }M\in\SO+B_\rho(0)\,.
\end{equation}
Let $\rho>0$ be such that \eqref{ok} holds true, let $\lambda>1$ be such that $\|\beta^{R^{\mathrm{T}}\xi^k,\C}_{\R^2}\|_{L^\infty(\R^2\setminus \overline{B}_{\lambda\ep}(0);\R^{2\times 2})}\le \frac{\rho}{\ep}$ and
let $\alpha\in(\gamma,1)$\,.
By the very definition of  $\beta_\ep $, using \eqref{iiprop} and \eqref{ok}, one can check that for every $k=1,\ldots,K$
\begin{multline}\label{dentrodentro}
\limsup_{\ep\to 0}\frac{1}{\ep^2|\log\ep|}\int_{A_{\lambda\ep,\ep^\alpha}(x^k)}W(\beta_\ep)\,\ud x\\
=
\limsup_{\ep\to 0}\frac{1}{\ep^2|\log\ep|}\int_{A_{\lambda\ep,\ep^\alpha}(x^k)}W(R_\ep(\mathrm{Id}+\ep\beta_{\R^2}^{R^{\mathrm{T}}\xi^k,\C}(x -x^k))\,\ud x\\
\le \limsup_{\ep\to 0}\frac{c_\rho}{|\log\ep|} \int_{A_{\lambda\ep,\ep^\alpha}(0)}\big|\beta^{R^{\mathrm{T}}\xi^k,\C}_{\R^2}\big|^2\,\ud x=C(1-\alpha)\,.
\end{multline}
Moreover, since $\|\ep\beta^{R^{\mathrm{T}}\xi^k,\C}_{\R^2}\|_{L^\infty}\le C\ep^{1-\alpha}$ in $A_{\ep^\alpha,\ep^{\gamma}}(0)$\,, by considering the Taylor expansion of $W$ around the identity, for any $k=1,\ldots,K$  we have
\begin{equation*}
\begin{aligned}
&\limsup_{\ep\to 0}\frac{1}{\ep^2|\log\ep|}\int_{A_{\ep^\alpha,\ep^{\gamma}}(x^k)}W(\beta_\ep) \,\ud x \\
=&\, 
\limsup_{\ep\to 0}\frac{1}{\ep^2|\log\ep|}\int_{A_{\ep^\alpha,\ep^{\gamma}}(x^k)}W(R_\ep(\mathrm{Id}+\ep\beta_{\R^2}^{R^{\mathrm{T}}\xi^k,\C}(x -x^k))\,\ud x
\\
=&
\limsup_{\ep\to 0}\frac{1}{|\log\ep|}\int_{A_{\ep^\alpha,\ep^{\gamma}}(0)}\frac 1 2\C\beta^{R^{\mathrm{T}}\xi^k,\C}_{\R^2}:\beta^{R^{\mathrm{T}}\xi^k,\C}_{\R^2} 
+ \frac{1}{\ep^2}\sigma(\ep\beta^{R^{\mathrm{T}}\xi^k,\C}_{\R^2})
\,\ud x\\
\end{aligned}
\end{equation*}
with $\lim_{|M|\to 0}\frac{\sigma(M)}{|M|^2}=0$\,. Therefore, in view of  \eqref{betapiano} and \eqref{servealimsup}, we get
\begin{equation}\label{dentro}
\limsup_{\ep\to 0}\frac{1}{\ep^2|\log\ep|}\int_{A_{\ep^\alpha,\ep^{\gamma}}(x^k)}W(\beta_\ep)\,\ud x\le(\alpha-\gamma)\psi^{\C}(R^{\mathrm{T}}\xi^k)\,.
\end{equation}
Furthermore, similar computations (or  arguing verbatim as in the proof of \cite[Proposition 3.12]{SZ} for all details)  show that
\begin{equation}\label{comeinSZ}
\limsup_{\ep\to 0}\frac{1}{\ep^2|\log\ep|}\int_{\Omega_{2\ep^{\gamma}}(\mu)}W(\beta_\ep)\,\ud x\le \gamma\sum_{k=1}^K\psi^{\C}(R^{\mathrm{T}}\xi^k)+\frac 1 2\int_{\Omega}\C R^{\mathrm{T}}\betalin:R^{\mathrm{T}}\betalin\,\ud x\,.
\end{equation}
Finally, by \eqref{peranello} and \eqref{ok}, for every $k=1,\ldots,K$ we have
\begin{equation*}\label{sullanello}
\frac{1}{\ep^2|\log\ep|}\int_{A_{\ep^{\gamma},2\ep^{\gamma}}(x^k)}W(\beta_\ep)\,\ud x\le \frac{c}{\ep^2|\log\ep|}\|\beta_\ep-R_\ep\|^2_{L^2(A_{\ep^{\gamma},2\ep^{\gamma}}(x^k);\R^{2\times 2})}\to 0\,,
\end{equation*}
which, combined with \eqref{dentrissimo}, \eqref{dentrodentro}, \eqref{dentro} and \eqref{comeinSZ}, implies \eqref{limsupformula} (sending $\alpha$ to $1$)\,. 
\end{proof}
\begin{remark}
\rm{
Let $0<\gamma\le \gamma'<1$\,. For every $\ep>0$\,, we can define the functional
 $\F^{\gamma,\gamma'}_\ep:\M(\R^2;\R^2)\times L^2(\Omega;\R^{2\times 2})\to [0,+\infty]$ as
\begin{equation*}
\F^{\gamma,\gamma'}_\ep(\mu,\beta):=\left\{\begin{array}{ll}
\int_{\Omega_\ep(\mu)}W(\beta)\,\ud x&\textrm{if }\mu\in X^\gamma_\ep(\Omega)\textrm{ and }\beta\in\AS^{\gamma'}_\ep(\mu)\\
+\infty&\textrm{otherwise.}
\end{array}
\right.
\end{equation*}
Clearly, $\F_\ep^{\gamma,\gamma}\equiv\F_\ep^{\gamma}$\,.
By following verbatim the proofs of Theorem \ref{teo:comp} and Theorem \ref{teo:Gamma-conv}, one can easily check that the same compactness and $\Gamma$-convergence statements hold true also for the functional $\F^{\gamma,\gamma'}_\ep$\,.

Furthermore, for every $\ep,\delta_\ep>0$ with $\delta_\ep<\ep^\gamma$ we can define the class $\AS_{\ep,\delta_\ep}(\mu)$ of admissible strains for a measure $\mu\in X_\ep^\gamma(\Omega)$ as in \eqref{adm}, with the condition $\fint_{A_{\ep,\ep^\gamma}(x^n)} \beta \, \ud x \in  R^n\mathcal I(\mathbb S)$ replaced by $\fint_{A_{\ep,\delta_\ep}(x^n)} \beta \, \ud x \in  R^n\mathcal I(\mathbb S)$\,; analogously, we can 
define the energy functional
\begin{equation*}
\F^{\gamma}_{\ep,\delta_\ep}(\mu,\beta):=\left\{\begin{array}{ll}
\int_{\Omega_\ep(\mu)}W(\beta)\,\ud x&\textrm{if }\mu\in X^\gamma_\ep(\Omega)\textrm{ and }\beta\in\AS_{\ep,\delta_\ep}(\mu)\\
+\infty&\textrm{otherwise.}
\end{array}
\right.
\end{equation*}
One can check that if $\delta_\ep\gg\ep\sqrt{|\log\ep|}$\,, the compactness and $\Gamma$-convergence results proved for the functional $\F_\ep^\gamma$ still hold true. As observed in Remark \ref{nonfunz}, if $\delta_\ep\sim\ep$ the coherence of the micro-rotations around each dislocation with the macroscopic rotation provided by linearization would fail.
}
\end{remark}
\section{The purely discrete model}
In this section we introduce  and analyze the nonlinear purely discrete model for the elastic energy induced by a family of edge dislocations.
\subsection{Description of the problem}
Here we introduce the main notation that will be used in this section.
\vskip5pt
\paragraph{\bf The reference lattice.} 
We set
$\nu:=\frac{1}{2} e_1+\frac{\sqrt 3}{2} e_2$ and $\eta:=-\frac 1 2 e_1+\frac{\sqrt 3}{2} e_2$\,.
Let $\Tl:=\Span_\Z\{e_1,\nu\}$ and set 
$$
T^+:=\conv\{0, e_1,\nu\}\quad\textrm{ and }\quad T^-:=\conv\{0, e_1,-\eta\}\,,
$$
where, for every $a,b,c\in\R^2$, the set $\conv\{a,b,c\}$ denotes the convex envelope of the points $a$, $b$, $c$, i.e., the (closed) triangle with vertices at $a$, $b$, $c$\,.
For every $\ep>0$ we denote by $\T_\ep$ the family of the triangles  $T_\ep$ of the form $i+\ep T^{\pm}$, with $i\in\ep\Tl$\,.
Moreover, we set
$$
\T_\ep(\Omega):=\{T_\ep\in\T_\ep\,:\,T_\ep\subset\Omega\}
$$ 
and 
we define
 $\insieme:=\bigcup_{T_\ep\in\T_\ep(\Omega)}T_\ep$. Furthermore, we set $\Omega_\ep^0:=\insieme\cap \ep\Tl$ and we denote by $\Omega_{\ep}^1$ the family of nearest neighbor bonds in $\insieme$, i.e., $\Omega_{\ep}^1:=\{(i,j)\in\Omega_\ep^0\times \Omega_\ep^0\,:\,|i-j|=\ep\}$. 
Trivially, $(i,j)\in\Omega_\ep^1$ if and only if $(j,i)\in\Omega_\ep^1$\,.
  
In the following we will generalize the notation introduced above to general subsets of $\R^2$ (not necessarily open). In particular, for every triangle $T_\ep\in\T_\ep$\,, we have
\begin{equation*}
(T_\ep)_\ep^1=\{(i,j)\in (T_\ep\cap \ep\Tl)\times (T_\ep\cap\ep\Tl)\,:\,i\neq j\}\,.
\end{equation*}
For every $T_\ep\in\T_\ep$ and for every
 map $V:(T_\ep)_\ep^1\to\R^2$\,, we define the {\it discrete circulation} of $V$ on the ``boundary of $T_\ep$'' as
\begin{equation*}
\ud V(T_\ep):=V(i,j)+V(j,k)+V(k,i)\,,
\end{equation*}
where $(i,j,k)$ is a triple of  counter-clockwise oriented vertices of $T_\ep$\,.
\vskip5pt
\paragraph{\bf The admissible strains and the energy functional}
For every $\ep>0$ we define the class of discrete strains as
\begin{equation*}
\sdi(\Omega):=\{\beta:\Omega_\ep^1\to\R^{2}\,:\,\beta(i,j)=-\beta(j,i)\textrm{ for any }(i,j)\in\Omega_\ep^1\}\,.
\end{equation*}
In the following, for every $\beta\in\sdi(\Omega)$\,, we set
\begin{equation*}
\mu[\beta]:=\sum_{T_\ep\in\T_\ep(\Omega)}\ud\beta(T_\ep)\delta_{x_{T_\ep}}\,,
\end{equation*}
where the point $x_{T_\ep}$ denotes the barycenter of the $\ep$-triangle $T_\ep$\,.
For any $\beta\in\sdi(\Omega)$ and for any $T_\ep\in\T_\ep(\Omega)$ with $\mu[\beta](T_\ep)=0$ we
 denote by $\tilde\beta^{T_\ep}$ the matrix in $\R^{2\times 2}$ uniquely defined by  the following property
 \begin{equation}\label{piecewisetri}
 \beta(i,j)=\tilde\beta^{T_\ep}(j-i)\qquad\textrm{for every }(i,j)\in (T_\ep)_\ep^1\,. 
 \end{equation}
 Moreover, we define the field $\tilde\beta^{\mathcal T_\ep}\in L^2(\Omega;\R^{2\times 2})$ as
\begin{equation}\label{piecewise}
\tilde\beta^{\mathcal{T}_\ep}:=\sum_{\newatop{T_\ep\in\T_\ep(\Omega)}{\mu[\beta](T_\ep)=0}}\tilde\beta^{T_\ep}\chi_{T_\ep\cap\Omega_\ep(\mu[\beta])}\,;
\end{equation}
notice that every $T_\ep\in\T_\ep(\Omega)$ with $\mu[\beta](T_\ep)\neq 0$ is contained in $B_\ep(x_{T_\ep})$\,. 

Let 
 $\psi_1, \psi_2\in C(\R;[0,+\infty))$ be such that $\psi_1^{-1}(0)=\psi_2^{-1}(0)=\{1\}$\,, $\psi_1,\psi_2\in C^2$ in a neighborhood of $1$\,,
 and $\psi_1''(1)\,,\psi_2''(1)>0$\,.  Assume moreover that there exist $a,b>0$ such that
$
\psi_1(t)\geq a t^2-b$\,, for every $t\in [0,+\infty)$\,.

We define the discrete energy functional $E_\ep:\sdi(\Omega)\to[0,+\infty)$ as
\begin{equation*}
\begin{aligned}
E_\ep(\beta):=&\sum_{(i,j)\in\Omega_\ep^1}\ep^2\psi_1\Big(\frac{|\beta(i,j)|}{\ep}\Big)\\
&+\sum_{\newatop{(i,j),(i,k)\in \Omega_\ep^1}{\langle (j-i)\wedge (k-i),e_3\rangle>0}}\ep^2\psi_2\Big(\frac{2}{\sqrt 3\ep^2}\langle\beta(i,j)\wedge\beta(i,k),e_3\rangle\Big)\,.
\end{aligned}
\end{equation*}
We will consider also localized versions of our energy functional $E_\ep$\,. More specifically, for every open and bounded set $A\subset\R^2$\,, we define $E_\ep(\cdot;A):\sdi(A)\to [0,+\infty]$ as 
\begin{equation*}
\begin{aligned}
E_\ep(\beta;A):=&\sum_{(i,j)\in A_\ep^1}\ep^2\psi_1\Big(\frac{|\beta(i,j)|}{\ep}\Big)\\
&+\sum_{\newatop{(i,j),(i,k)\in A_\ep^1}{\langle (j-i)\wedge (k-i),e_3\rangle>0}}\ep^2\psi_2\Big(\frac{2}{\sqrt 3\ep^2}\langle\beta(i,j)\wedge\beta(i,k),e_3\rangle\Big)\,,
\end{aligned}
\end{equation*}
so that $E_\ep(\beta;\Omega)=E_\ep(\beta)$\,.

Let $0<\gamma<1$ be fixed. 
We set $S:=\{e_1,\nu\}$\,, so that, following the notation in Section \ref{sec:semidi}, $\Ss=\Tl$\,.
We define
\begin{equation*}
\mdi:=\Big\{\mu\in X_\ep^\gamma(\Omega)\,:\,\supp\mu\subset\bigcup_{T_\ep\in\T_\ep(\Omega)}\{x_{T_\ep}\}\Big\}\,,
\end{equation*}
where the class $X^\gamma_\ep(\Omega)$ is defined in \eqref{mcont}.

Moreover, for any $\mu=\ep\sum_{n=1}^N  R^n b^n \delta_{x^n}\in\mdi$\,, with $N\in\N$\,, $R^n\in\SO$\,, $b^n\in\Tl$\,, and $x^n\in\Omega$\,, we define the class of admissible discrete strains associated to $\mu$ as
\begin{equation*}
\begin{aligned}
\asdi(\mu):=\Big\{\beta\in\sdi(\Omega)\,:\,
\fint_{A_{\ep,\ep^\gamma}(x^n)} \tilde\beta^{\mathcal{T}_\ep} \, \ud x \in R^n\mathcal I(\Tl), \quad
\mu[\beta]= \mu 
\Big\}\,,
\end{aligned}
\end{equation*}
where $\tilde \beta^{\mathcal{T}_\ep}$ is the map defined in \eqref{piecewise} and $\mathcal{I}(\Tl)$ is the group of rotations generated by the $\frac\pi 3$ clockwise rotation $R(\frac\pi 3)$\,. 

 For every $0<\ep<1$
we define the discrete energy $\E_\ep^\gamma:\M(\R^2;\R^2)\times\sdi(\Omega)\to [0,+\infty]$ as
\begin{equation*}
\E^{\gamma}_\ep(\mu,\beta):=\left\{\begin{array}{ll}
E_\ep(\beta)&\textrm{if }\mu\in \mdi\textrm{ and }\beta\in \asdi(\mu)\,,\\
+\infty&\textrm{otherwise.}
\end{array}\right.
\end{equation*}
Before stating and prove our $\Gamma$-convergence result for the functional $\E^{\gamma}_\ep$\,, we first derive the continuous nonlinear elastic energy density associated
to the discrete functional $E_\ep$ as well as the corresponding linearized elasticity tensor.
{
\begin{remark}\label{interpred}
\rm{
Let $\beta\in\sdi(\Omega)$\,; 
for every  $T_\ep\in\T_\ep(\Omega)$ we define
\begin{equation*}
\begin{aligned}
\tilde E_\ep(\beta;T_\ep):=\frac 12&\sum_{(i,j)\in (T_\ep)_\ep^1}\ep^2\psi_1\Big(\frac{|\beta(i,j)|}{\ep}\Big)\\
&\,+ \sum_{\newatop{(i,j),(i,k)\in (T_\ep)_\ep^1}{\langle (j-i)\wedge (k-i),e_3\rangle>0}}\ep^2\psi_2\Big(\frac{2}{\sqrt 3\ep^2}\langle\beta(i,j)\wedge\beta(i,k),e_3\rangle\Big)\,.
\end{aligned}
\end{equation*}
Note that, for every $A\subset\Omega$\,,
\begin{equation}\label{energiariscritta0}
E_\ep(\beta;A)=\sum_{\newatop{T_\ep\in\T_\ep}{T_\ep\subset A_{\T_\ep}}}\tilde{E}_\ep(\beta;T_\ep)+\frac{1}{2}\sum_{\newatop{(i,j)\in A_\ep^1}{i,j\in\partial A_{\mathcal{T}_\ep}}}\ep^2\psi_1\Big(\frac{|\beta(i,j)|}{\ep}\Big)\,.
\end{equation}}
Moreover, for any $T_\ep\in\T_\ep(\Omega)$ with $\mu[\beta](T_\ep)=0$\,, we easily get that
\begin{equation}\label{energia_tri}
\begin{aligned}
\tilde E_\ep(\beta;T_\ep)=&\, \ep^2\Big[\frac{1}{2}\Big(\psi_1(|\tilde\beta^{T_\ep}e_1|)+\psi_1(|\tilde \beta^{T_\ep}\nu|)+\psi_1(|\tilde\beta^{T_\ep}\eta|)\Big)
+3\, \psi_2(\det\tilde\beta^{T_\ep})\Big]\\
=&\,\frac{\sqrt 3}{4}\ep^2 W(\tilde\beta^{T_\ep})=\int_{T_\ep}W(\tilde\beta^{T_\ep})\,\ud x\,,
\end{aligned}
\end{equation}
where $\tilde\beta^{T_\ep}$ is the matrix defined in \eqref{piecewisetri} and $W$ is given by
\begin{equation}\label{defW}
W(M):=\frac{4}{\sqrt{3}}\Big[\frac{1}{2}\Big(\psi_1(| M e_1|)+\psi_1(|M\nu|)+\psi_1(| M \eta|)\Big)+3\, \psi_2(\det M)\Big]\,.\\
\end{equation}
The assumptions on $\psi_1,\psi_2$ easily yield  that the map $W:\R^{2\times 2}\to \R$ satisfies the assumptions \eqref{iprop}-\eqref{vprop}.  Setting $\alpha_1:=\psi_1''(1)$ and $\alpha_2:=\psi_2''(1)$\,,  the corresponding tensor $\C$ in \eqref{tlin} is given by
\begin{equation}\label{tuttilame}
\C \delta:\delta:=\frac{4}{\sqrt{3}}\Big[\frac{1}{2}\alpha_1\Big(|e_1^* \delta e_1|^2 +|\nu^* \delta \nu|^2 +|\eta^* \delta \eta|^2 \Big)+3\alpha_2|\mathrm{tr}\, \delta|^2\Big]\,,\qquad\textrm{ for every }\delta\in\R^{2\times 2}\,.
\end{equation}
A straightforward computation shows that
\begin{equation*}
\begin{aligned}
\C \delta:\delta=&\frac{4}{\sqrt{3}}\Big[\frac{1}{2}\alpha_1\Big(\frac 38 |\mathrm{tr}\, \delta|^2+\frac34 |\delta^{\sym}|^2 \Big)+3\alpha_2|\mathrm{tr}\, \delta|^2\Big]\\
=& \Big(\frac{\sqrt3}{4}\alpha_1+4\sqrt{3}\alpha_2\Big)|\mathrm{tr}\, \delta|^2+\frac{\sqrt3}{2}\alpha_1|\delta^{\sym}|^2 .
\end{aligned}
\end{equation*}
In particular, $\C$ is isotropic with Lam\'e moduli $\lambda=\frac{\sqrt3}{4}\alpha_1+4\sqrt{3}\alpha_2$ and $\mu=\frac{\sqrt3}{4}\alpha_1$\,.
In such a case, the function $\psi^{\C}$ defined in \eqref{servealimsup} is given (see for instance \cite{CL})  by
\begin{equation*}
\psi^{\C}(b)=\frac{1}{4\pi}\frac{\mu(\lambda+\mu)}{\lambda+2\mu}|b|^2
=:C(\alpha_1,\alpha_2)|b|^2\,,
\end{equation*}
and hence, recalling the definition of $\ffi^\C$ in \eqref{generffi}, we have 
\begin{equation}\label{ffiora}
\ffi^{\C}(b)=
C(\alpha_1,\alpha_2)\min\bigg\{\sum_{i=1}^{3}|z_i|\,:\, z_1,z_2,z_3\in\Z\,,\, b=z_1e_1+z_2\nu+z_3\eta\bigg\}\,.
\end{equation}
\end{remark}
By \eqref{energia_tri}, \eqref{piecewise} and \eqref{energiariscritta0}, for each $\mu\in\mdi$ and $\beta\in\sdi(\Omega)$ with $\mu[\beta]=\mu$, and for every open set $A\subset\subset\Omega$\,, for $\ep$ sufficiently small it holds
\begin{equation}\label{energiariscritta}
E_\ep(\beta)\ge\int_{A_\ep(\mu)}W(\tilde\beta^{\mathcal{T}_\ep})\,\ud x\,.
\end{equation}
By a reflection argument (see, for instance, \cite[Lemma 4.3]{ADLPP}), one can prove the following result.
\begin{lemma}\label{ext}
Let $\ep>0$\,, $\mu\in\mdi$ and $\beta\in\asdi(\mu)$ and let $\tilde\beta^{\mathcal{T}_\ep}$ be the map defined in \eqref{piecewise}. Then, there exists a field $\hat\beta^{\mathcal{T}_\ep}\in L^2(\Omega;\R^{2\times 2})$ 
such that
\begin{itemize}
\item[(i)] $\hat\beta^{\mathcal{T}_\ep}=\tilde\beta^{\mathcal{T}_\ep}$ in $\insieme$\,;
\item[(ii)] $\Cu\hat\beta^{\mathcal{T}_\ep}=0$ in $\Omega_\ep(\mu)$ (in the sense of distributions);
\item[(iii)] $\int_{\Omega_\ep(\mu)}W(\hat\beta^{\mathcal{T}_\ep})\,\ud x\le C E_\ep(\beta)$\,, for some constant $C$ independent of $\ep$\,.
\end{itemize}
\end{lemma}
\subsection{The main result for the discrete energy}
We are now in a position to state and prove our $\Gamma$-convergence result for the functionals $\E_\ep^{\gamma}$ as $\ep\to 0$ in the $|\log\ep|$ regime.
\begin{theorem}\label{gammaconvdiscr}
Let $\C$ be the elasticity tensor defined by \eqref{tuttilame} and let $\ffi^\C$ be the function defined in \eqref{ffiora}.
The following $\Gamma$-convergence result holds true.
\begin{itemize}
\item[(i)]
Let $\{(\mu_\ep;\beta_\ep)\}_\ep$\,,
with $\mu_\ep\in\M(\R^2;\R^2)$ and $\beta_\ep\in\sdi(\Omega)$ (for every $\ep>0$), satisfy  
\begin{equation}\label{enbounddiscr}
\sup_{\ep>0}\E^{\gamma}_\ep(\mu_\ep,\beta_\ep)\le C\ep^2|\log\ep|\,,
\end{equation}
for some constant $C>0$\,.
Then, there exist a sequence of rotations $\{R_\ep\}_\ep\subset\SO$\,, a rotation $R\in\SO$\,, a measure $\mu\in \overline X^R(\Omega)$\,, and a field $\betalin\in L^2(\Omega;\R^{2\times 2})$ with $\Cu\betalin=0$\,,
 such that, up to a subsequence, 
$(\mu_\ep;\tilde\beta^{\mathcal{T}_\ep}_\ep;R_\ep)\to (\mu;\betalin;R)$ as $\ep\to 0$ in the sense of Definition \ref{conve}.
\item[(ii)] ($\Gamma$-liminf inequality) Let $R\in\SO$\,, $\mu=\sum_{k=1}^K\xi^k\delta_{x^k}\in \overline{X}^R(\Omega)$\,, $\betalin\in L^2(\Omega;\R^{2\times 2})$ with $\Cu\betalin=0$\,. For every $\{(\mu_\ep;\beta_\ep;R_\ep)\}_\ep$\,, with 
$\mu_\ep\in\mdi$\,, $\beta_\ep\in\sdi(\Omega)$ and $R_\ep\in\SO$ (for every $\ep>0$), satisfying
$(\mu_\ep;\tilde\beta^{\mathcal{T}_\ep}_\ep;R_\ep)\to (\mu;\betalin;R)$ as $\ep\to 0$\,, 
it holds
\begin{equation}\label{liminfformuladiscr}
\sum_{k=1}^K\ffi^{\C}(R^{\mathrm{T}}\xi^k)+\int_{\Omega}\frac 1 2\C R^{\mathrm{T}}\betalin:R^{\mathrm{T}}\betalin\,\ud x\le\liminf_{\ep\to 0}\frac{1}{\ep^2|\log\ep|}\E_\ep^{\gamma}(\mu_\ep,\beta_\ep)\,.
\end{equation}
\item[(iii)] ($\Gamma$-limsup inequality) Let $R\in\SO$\,, $\mu=\sum_{k=1}^K\xi^k\delta_{x^k}\in \overline{X}^R(\Omega)$ and $\betalin\in L^2(\Omega;\R^{2\times 2})$ with $\Cu\betalin=0$\,.
 Then, there exists $\{(\mu_\ep;\beta_\ep;R_\ep)\}_\ep$\,, with 
$\mu_\ep\in\mdi$\,, $\beta_\ep\in\sdi(\Omega)$ and $R_\ep\in\SO$ (for every $\ep>0$), 
 such that $(\mu_\ep;\tilde\beta^{\mathcal{T}_\ep}_\ep;R_\ep)\to (\mu;\betalin;R)$ as $\ep\to 0$ 
 and
\begin{equation}\label{limsupformuladiscr}
\sum_{k=1}^K\ffi^{\C}(R^{\mathrm{T}}\xi^k)+\int_{\Omega}\frac 1 2\C R^{\mathrm{T}}\betalin:R^{\mathrm{T}}\betalin\,\ud x\ge\limsup_{\ep\to 0}\frac{1}{\ep^2|\log\ep|}\E_\ep^{\gamma}(\mu_\ep,\beta_\ep)\,.
\end{equation}
\end{itemize}
\end{theorem}
\begin{proof}
We start by proving the compactness property (i).
In view of \eqref{enbounddiscr}, the maps $\hat\beta^{\mathcal{T}_\ep}_\ep$ provided by Lemma \ref{ext} belong to $\AS^{\gamma}_\ep(\mu_\ep)$ and satisfy
\begin{equation*}
\F_\ep^{\gamma}(\mu_\ep,\hat\beta^{\mathcal{T}_\ep}_\ep)\le C\ep^2|\log\ep|\,,
\end{equation*} 
where the functional $\F_\ep^{\gamma}$ is defined in \eqref{enercont} for the choice of $W$ in \eqref{defW}.
Therefore, the conclusion follows by  Theorem \ref{teo:comp} and Lemma \ref{ext}(i), once noticed that $\|R_\ep\|_{L^2(\Omega\setminus\insieme)}\le C\ep$\,.
\vskip5pt
Now we pass to the proof of (ii). Without loss of generality we may assume that \eqref{enbounddiscr} is satisfied. 
In view of \eqref{dueprop} we have that $|\frac{\mu_\ep}{\ep}|(\Omega)\le C$ for some $C$ independent of $\ep$\,, so that,
 up to passing to a subsequence, $|\frac{\mu_\ep}{\ep}|\weakstar\bar\mu$ for some measure $\bar\mu\in\M(\R^2;\R^2)$ with finite support contained in $\Omega$\,. 
Let $A\subset\subset\Omega$ be an open set with Lipschitz continuous boundary such that $\supp\bar\mu\subset A$\,.
 By construction,
 for sufficiently small $\ep>0$ we have that
  $\mu_\ep\res A\in X_\ep^{\gamma}(A)$ and $\tilde\beta^{\mathcal{T}_\ep}_\ep\res A\in \AS_\ep^{\gamma}(A)$\,.
 Then  \eqref{liminfformuladiscr} follows by \eqref{energiariscritta}, Theorem \ref{teo:Gamma-conv}(i) applied with $\Omega=A$\,, and by the arbitrariness of $A$\,. 

\vskip5pt
Finally, we prove (iii). By standard density arguments in $\Gamma$-convergence we can assume that $\ffi^\C(R^{\mathrm{T}}\xi^k)=\psi^\C(R^{\mathrm{T}}\xi^k)$ for every $k=1,\ldots,K$ and that $\betalin\in W^{1,\infty}(\Omega;\R^{2\times 2})$\,. Let furthermore $u^{\betalin}\in W^{2,\infty}(\Omega;\R^2)$ be such that $\betalin=\nabla u^{\betalin}$\,.

For every $k=1,\ldots,K$ let 
$$
x^{k,\ep}\in\mathrm{argmin}\{|x^k-x_{T_\ep}|\,:\, T_\ep\in\T_\ep(\Omega)\}
$$
and set $\mu_\ep:=\ep\sum_{k=1}^{K}\xi^k\delta_{x^{k,\ep}}$\,. Then, for $\ep$ small enough, $\mu_\ep\in\mdi$\,.
Moreover, we set $R_\ep\equiv R$\,.
In order to construct $\beta_\ep$\,, we proceed as in the proof of Theorem \ref{teo:Gamma-conv}(ii). We first construct $u_\ep$ as in \eqref{recose}, replacing $x^k$ with $x^{k,\ep}$ and suitably modifying the radii of the annuli. 

More precisely, for every $k=1,\ldots,K$ we define the map $u_\ep^{R,k}\in C^{\infty}(\R^2\setminus\cut^{k,\ep};\R^2)$ as
\begin{equation*}
u_\ep^{R,k}(x):=u^{R,R^{\mathrm{T}}\xi^k}(x-x^{k,\ep})+\sqrt{|\log\ep|}u^{\betalin}(x^{k,\ep})\,,
\end{equation*}
with $\cut^{k,\ep}:=x^{k,\ep}+\cut$ (with $\cut:=\{(x_1;0)\,:\,x_1\ge 0\}$)
and $u^{R,\zeta}$ satisfying \eqref{23marzo}.
Therefore, $\nabla u_\ep^{R,k}(\cdot)=R\beta^{R^{\mathrm{T}}\xi^{k},\C}_{\R^2}(\cdot-x^{k,\ep})$ and $[u^{R,k}_\ep]=\xi^k$ on $\cut^{k,\ep}$ (locally in the sense of traces).

We define the map $u_\ep^{\far}:\Omega_{\ep}(\mu_\ep)\to\R^{2}$ as
\begin{equation*}
 u_\ep^{\far}(x):=\sum_{k=1}^Ku^{R,R^{\mathrm{T}}\xi^k}(x-x^{k,\ep})+\sqrt{|\log\ep|}u^{\betalin}(x)\,.
\end{equation*}
Let $\phi\in C^1([0,3];[0,1])$ be such that $\phi\equiv1$ in $[0,1]$ and $\phi\equiv 0$ in $[2,3]$. We define the function $u_\ep:\Omega_{\frac{\ep}{2}}(\mu_\ep)\to\R^2$ as 
\begin{equation*}
u_\ep(x):=\left\{
\begin{array}{ll}
u_\ep^{R,k}(x)&\textrm{if }x\in A_{\frac{\ep}{2},\ep^{\gamma}+\ep}(x^{k,\ep})\textrm{ for some }k\,,\\
\phi\big(\frac{|x-x^{k,\ep}|}{\ep^{\gamma}+\ep}\big)u_\ep^{R,k}(x)+\big(1-\phi\big(\frac{|x-x^{k,\ep}|}{\ep^{\gamma}+\ep}\big)\big)u_\ep^{\far}(x)&\textrm{if }x\in A_{\ep^{\gamma}+\ep,2(\ep^{\gamma}+\ep)}(x^{k,\ep})\textrm{ for some }k\,,\\
u_\ep^{\far}(x)&\textrm{if }x\in\Omega_{2(\ep^{\gamma}+\ep)}(\mu)\,.
\end{array}
\right.
\end{equation*}
Furthermore, for every $k=1,\ldots,K$ we define the slip variable $\sigma^{k,\ep}:\Omega_\ep^1\to \ep\Tl$ as
\begin{equation*}
\sigma^{k,\ep}(i,j):=\left\{
\begin{array}{ll}
- \xi^k&\textrm{if }\cut^{k,\ep}\cap [i,j]\neq\emptyset \textrm{ and }\langle i-x^{k,\ep},e_2\rangle<0\\
+\xi^k&\textrm{if }\cut^{k,\ep}\cap [i,j]\neq\emptyset\textrm{ and }\langle i-x^{k,\ep},e_2\rangle>0\\
0&\textrm{elsewhere},
\end{array}
\right.
\end{equation*}
and we set $\sigma_\ep:=\sum_{k=1}^{K}\sigma^{k,\ep}$\,.
Finally, we define the map $\beta_\ep\in\sdi(\Omega)$ as
\begin{equation*}
\beta_\ep(i,j):=R(j-i)+\ep(u_\ep(j)-u_\ep(i)-\sigma_{\ep}(i,j))\,.
 \end{equation*}
 We claim that $\{(\mu_\ep;\beta_\ep;R_\ep)\}_\ep$ is a recovery sequence. By construction, $\mu_\ep=\mu[\beta_\ep]$ and,
 by Remark \ref{zeromean} (in particular, by \eqref{precirem}), we have also that $\beta_\ep\in\asdi(\mu_\ep)$ for every $\ep>0$\,.
 Moreover, properties \eqref{zeroprop} and \eqref{dueprop} are trivially satisfied. 
 We show that this is the case also for \eqref{unoprop}.
 By the very definition of $\beta^{\zeta,\C}_{\R^2}$ in \eqref{betapiano}, there exists a universal constant $C>0$ such that for every $\zeta\in\Tl$
 \begin{equation}\label{stimasf}
 |\beta^{\zeta,\C}_{\R^2}(\rho,\theta)|\le \frac{C}{\rho}|\zeta|^2\,,\qquad  |\nabla\beta^{\zeta,\C}_{\R^2}(\rho,\theta)|\le \frac{C}{\rho^2}|\zeta|^2\,.
 \end{equation}
 Therefore, recalling that $\betalin\in W^{1,\infty}(\Omega;\R^{2\times 2})$\,, for every $T_\ep\in\T_\ep(\Omega)$ with $T_\ep\subset \bigcup_{k=1}^KA_{\frac\ep 2,\ep^{\gamma}+\ep}(x^{k,\ep})\cup\Omega_{2(\ep^\gamma+\ep)}(\mu_\ep)$ and for every $x\in T_\ep$ it holds
 \begin{equation}\label{claimfond}
 \begin{aligned}
 \Big|\frac{\tilde\beta_{\ep}^{T_\ep}-R_\ep}{\ep}-\nabla u_\ep(x)\Big|^2\le &\,C\ep^2\|\nabla^2 u_\ep\|^2_{L^\infty(T_\ep;\R^{2\times 2})}\\
 \le&\, C\ep^2\Big( \sum_{k=1}^{K}\|\nabla\beta^{R^{\mathrm{T}}\xi^k,\C}_{\R^2}\|^2_{L^\infty(T_\ep;\R^{2\times 2})}+|\log\ep|\|\nabla\betalin\|^2_{L^\infty(T_\ep;\R^{2\times 2})}\Big)
\\
\le&\, C\ep^2\Big( \sum_{k=1}^{K}\frac{1}{|x-x^{k,\ep}|^4}+|\log\ep|\Big)\,,
 \end{aligned}
 \end{equation}
 for some constant $C>0$ independent of $\ep$\,.
 Furthermore, by arguing verbatim as in the proof of \eqref{peranello0}, we have that, for every $k=1,\ldots,K$\,,
 \begin{equation}\label{panello}
\|\nabla u_\ep\|_{L^\infty\big(A_{\frac{\ep^\gamma}{2},3\ep^\gamma}(x^{k,\ep});\R^{2\times 2}\big)}+\Big\|\frac{\tilde\beta^{\mathcal{T}_\ep}_\ep-R_\ep}{\ep}\Big\|_{L^\infty\big(A_{\frac{\ep^\gamma}{2},3\ep^\gamma}(x^{k,\ep});\R^{2\times 2}\big)}\le \frac{C}{\ep^\gamma}\,.
 \end{equation}
Therefore, by \eqref{claimfond}-\eqref{panello} and the fact that $|\Omega\setminus\insieme|\le C\ep$\,, we readily see that 
\begin{equation}\label{stessanorma}
\frac{1}{|\log\ep|} \Big\|\frac{\tilde\beta^{\mathcal{T}_\ep}_{\ep}-R_\ep}{\ep}-\nabla u_\ep\chi_{\Omega_\ep(\mu_\ep)}\Big\|_{L^2(\Omega;\R^{2\times 2})}^2\to 0\qquad\textrm{as }\ep\to 0.
\end{equation}
 Moreover, by arguing verbatim as in \eqref{peranello00}, \eqref{peranello} and \eqref{finale}, we have
 \begin{equation}\label{quasiladue}
 \frac{\nabla u_\ep\chi_{\Omega_\ep(\mu_\ep)}}{\sqrt{|\log\ep|}}\weakly \betalin\qquad\textrm{in }L^2(\Omega;\R^{2\times 2})\,.
 \end{equation}
 Therefore, by \eqref{stessanorma} and \eqref{quasiladue}, we deduce that $(\mu_\ep;\tilde\beta_\ep^{\mathcal{T}_\ep};R_\ep)\to (\mu;\betalin;R)$\,.
 
Next we show that also \eqref{limsupformuladiscr} holds true. 
 By construction 
  we have
\begin{equation}\label{stimacore}
\tilde E_\ep(\beta_\ep;T_\ep)\le C\ep^2\qquad\textrm{for every }T_\ep\in\T_\ep(\Omega)\textrm{ with }\min_{k=1,\ldots,K}\mathrm{dist}(x^{k,\ep},T_\ep)\le \ep\,,
\end{equation}
and that for every $(i,j)\in\Omega_\ep^1$ with $i,j\in\partial\insieme$
\begin{equation}\label{stimabordo}
\ep^2\psi\Big(\frac{|\beta_\ep(i,j)|^2}{\ep^2}\Big)\le C\ep^2\,.
\end{equation}
Moreover, by \eqref{panello} and \eqref{ok}, for every $k=1,\ldots,K$\,, we have
\begin{equation}\label{panello1}
\frac{1}{\ep^2|\log\ep|}\int_{A_{\frac{\ep^\gamma}{2},3\ep^\gamma}(x^{k,\ep})}W(\tilde\beta_\ep^{\mathcal{T}_\ep})\,\ud x\le\frac{1}{|\log\ep|} \Big\|\frac{\tilde\beta_\ep^{\mathcal{T}_\ep}-R_\ep}{\ep}\Big\|^2_{L^2\big(A_{\frac{\ep^\gamma}{2},3\ep^\gamma}(x^{k,\ep});\R^{2\times 2}\big)}\le \frac{C}{|\log\ep|}\,.
\end{equation}
Let  $\alpha\in(\gamma,1)$.
By arguing as in the proof of \eqref{dentrissimo} and \eqref{dentrodentro}, one can show that
\begin{equation}\label{dentrodentrod}
\limsup_{\ep\to 0}\frac{1}{\ep^2|\log\ep|}\int_{A_{{\ep},\ep^\alpha}(x^{k,\ep})}W(\tilde\beta^{\mathcal{T}_\ep}_\ep)\,\ud x
=C(1-\alpha)\,,
\end{equation}
for each $k=1,\ldots,K$.
Moreover, since $\|\tilde\beta_\ep^{\mathcal{T}_\ep}-R_\ep\|_{L^\infty}\le C\ep^{1-\alpha}$ in $A_{\ep^\alpha,\ep^{\gamma}}(x^{k,\ep})$\,, by \eqref{iiprop}, 
writing the Taylor expansion of $W$ around the identity, 
we find
\begin{equation}\label{oggi}
\begin{aligned}
&\limsup_{\ep\to 0}\frac{1}{\ep^2|\log\ep|}\int_{A_{\ep^\alpha,\ep^{\gamma}}(x^{k,\ep})}W(\tilde\beta^{\mathcal{T}_\ep}_\ep) \,\ud x \\
\le&\,
\limsup_{\ep\to 0}\frac{1}{|\log\ep|}\int_{A_{\ep^\alpha,\ep^{\gamma}}(x^{k,\ep})}\frac 1 2\C \Big(R^{\mathrm{T}}\frac{\tilde\beta^{\mathcal{T}_\ep}_\ep-R}{\ep}\Big): \Big(R^{\mathrm{T}}\frac{\tilde\beta^{\mathcal{T}_\ep}_\ep-R}{\ep}\Big)
\,\ud x\\
&\,+\limsup_{\ep\to 0}\frac{1}{|\log\ep|}\int_{A_{\ep^\alpha,\ep^{\gamma}}(x^{k,\ep})}\frac{\sigma(\tilde\beta_\ep^{\mathcal{T}_\ep}-R)}{|\tilde\beta_\ep^{\mathcal{T}_\ep}-R|^2}\Big|\frac{\tilde\beta_\ep^{\mathcal{T}_\ep}-R}{\ep}\Big|^2\,\ud x\,,
\end{aligned}
\end{equation}
with $\lim_{|M|\to 0}\frac{\sigma(M)}{|M|^2}=0$\,. By \eqref{claimfond} and \eqref{stimasf}, we have that
$\Big|\frac{\tilde\beta_\ep^{\mathcal{T}_\ep}-R}{\ep}\Big|\le \frac{C}{|x-x^{k,\ep}|}$ for every $x\in A_{\ep^\alpha,\ep^{\gamma}}(x^{k,\ep})$\,, whence we deduce that the last $\limsup$ in \eqref{oggi} is equal to zero. Therefore, by \eqref{oggi}, using again \eqref{claimfond}, we get
\begin{equation}\label{oggi2}
\begin{aligned}
&\limsup\frac{1}{\ep^2|\log\ep|}\int_{A_{\ep^\alpha,\ep^{\gamma}}(x^{k,\ep})}W(\tilde\beta^{\mathcal{T}_\ep}_\ep) \,\ud x\\
\le&\,\limsup_{\ep\to 0}\frac{1}{|\log\ep|}\int_{A_{\ep^\alpha,\ep^{\gamma}}(x^{k,\ep})}\frac 1 2\C \big(R^{\mathrm{T}}\nabla u_\ep\big):\big(R^{\mathrm{T}}\nabla u_\ep\big) +C\ep^2\Big(\frac{1}{|x-x^{k,\ep}|^4}+|\log\ep|\Big)\,\ud x\\
=&\, \limsup_{\ep\to 0}\frac{1}{|\log\ep|}\int_{A_{\ep^\alpha,\ep^{\gamma}}(0)}\frac 1 2\C\beta^{R^{\mathrm{T}}\xi^k,\C}_{\R^2}:\beta^{R^{\mathrm{T}}\xi^k,\C}_{\R^2}\,\ud x\\
=&\,(\alpha-\gamma)\psi^{\C}(R^{\mathrm{T}}\xi^k)\,,
\end{aligned}
\end{equation}
where the last equality follows by \eqref{servealimsup}. 
Furthermore, similar arguments yield
\begin{equation}\label{fuorid}
\limsup_{\ep\to 0}\frac{1}{\ep^2|\log\ep|}\int_{\Omega_{2\ep^\gamma}(\mu_\ep)}W(\tilde\beta_\ep^{\mathcal{T}_\ep})\,\ud x\\
\le\gamma\sum_{k=1}^K\psi^{\C}(R^{\mathrm{T}}\xi^k)+\frac 1 2\int_{\Omega}\C R^{\mathrm{T}}\betalin:R^{\mathrm{T}}\betalin\,\ud x\,.
\end{equation}
Finally, combining \eqref{energiariscritta0}, \eqref{stimacore}-\eqref{dentrodentrod}, \eqref{oggi2}-\eqref{fuorid}, and sending $\alpha\to 1$\,, we obtain \eqref{limsupformuladiscr}.
\end{proof}


\begin{thebibliography}{99}
\bibitem{AC}
Alicandro, R., Cicalese, M.: Variational analysis of the asymptotics of the XY model. {\it Arch. Rational Mech.
Anal.} {\bf 192} (2009), 501--536.

\bibitem{ACP}
Alicandro, R., Cicalese, M., Ponsiglione, M.: Variational equivalence between Ginzburg-Landau, XY spin
systems and screw dislocations energy. {\it Indiana Univ. Math. J.} {\bf 6} (2011), 171--208. 

\bibitem{ADGP} Alicandro, R., De Luca, L., Garroni, A., Ponsiglione: Metastability and dynamics of discrete topological singularities
in two dimensions: A $\Gamma$-convergence approach. {\it Arch. Rational Mech. Anal.} \textbf{214} (2014), 269--330.

\bibitem{ADLPP} Alicandro, R., De Luca, L., Lazzaroni, G., Palombaro, M., Ponsiglione, M.: Coarse-graining of e discrete model for edge dislocations in the regular triangular lattice. {\it J. Nonlinear Sci.} {\bf 33} (2023), art. n. 33.

\bibitem{ALP}  Alicandro, R., Lazzaroni, G., Palombaro, M.:
Derivation of linear elasticity for a general class of atomistic energies. \textit{SIAM J. Math. Anal.} \textbf{53} (2021), 5060--5093.

\bibitem{AO}  Ariza, M. P.; Ortiz, M.: Discrete crystal elasticity and discrete dislocations in crystals. {\it Arch. Rational Mech. Anal.} \textbf{178} (2005), 149--226.


\bibitem{BBS} Bacon, D.J., Barnett, D.M., Scattergood, R.O.: Anisotropic continuum theory of
lattice defects. {\it Progress Mater. Sci.} {\bf 23} (1978), 51--262. 

\bibitem{BSV} Braides, A., Solci, M., Vitali, E.: A derivation of linear elastic energies from pair-interaction atomistic systems.
\textit{Netw. Heterog. Media} \textbf{2}  (2007), 551--567.

\bibitem{CL}{Cermelli, P., Leoni, G.}: {Renormalized energies and forces on dislocations}. {\it SIAM J. Math. Anal.} {\bf 37} (2005), 1131--1160.

\bibitem{CGM}{Conti, S., Garroni, A., Marziani, R.}: {Line-tension limits for line singularities and application to the mixed-growth case}. {\it Preprint} (2022), ArXiv: 2207.01526.

\bibitem{DMNP}{Dal Maso, G., Negri, M., Percivale, D.}: Linearized Elasticity as $\Gamma$-limit of Finite Elasticity. {\it Set-Valued Anal.} \textbf{10} (2002), 165--183.

\bibitem{DL} De Luca, L: $\Gamma$-convergence analysis for discrete topological singularities: The anisotropic triangular lattice and the long range interaction energy. {\it Asymptot. Anal.} \textbf{96} (2016), 185--221.

\bibitem{DGP}
De Luca, L., Garroni, A., Ponsiglione, M.: $\Gamma$-convergence analysis of systems of edge dislocations: the self energy regime. {\it Arch. Rational Mech. Anal.} {\bf  206} (2012), 885--910. 



\bibitem{FJM}
{Friesecke, G., James, R.D., M\"uller, S.}: A theorem on geometric rigidity and the derivation of nonlinear plate theory from three-dimensional elasticity. {\it Comm. Pure Appl. Math.} {\bf 55} (2002), 1461--1506.

\bibitem{GLP} {Garroni, A., Leoni, G., Ponsiglione, M.}:  Gradient theory for plasticity via homogenization
of discrete dislocations. {\it J. Eur. Math. Soc.} \textbf{12} (2010),1231--1266.


\bibitem{GT}{Giuliani, A., Theil, F.}: Long range order in atomistic models for solids, to appear in {\it J. Eur. Math. Soc.}

\bibitem{HL} {Hirth, J.P., Lothe, J.}: {\it Theory of Dislocations}, Krieger Publishing Company, Malabar, Florida,
1982.


\bibitem{LM}
{S. Luckhaus, L. Mugnai}: On a mesoscopic many-body Hamiltonian describing elastic shears and dislocations.
{\it Cont. Mech. Thermodyn.} {\bf 22} (2010), 251--290.




\bibitem{MSZlibro}{M\"uller, S., Scardia, L., Zeppieri, C.I.}: Gradient theory for geometrically
nonlinear plasticity via the homogenization of dislocations.
In S. Conti and K. Hackl, editors, {\it Analysis and Computation
of Microstructure in Finite Plasticity}, 175--204. Springer,
2015.

\bibitem{P}
{Ponsiglione, M.}: Elastic energy stored in a crystal induced by screw dislocations: from discrete to continuous.
{\it SIAM J. Math. Anal.} {\bf 39} (2017), 449--469.

\bibitem{RC}{Reina, C., Conti, S.}: Kinematic description of crystal plasticity in the finite kinematic framework, a micromechanical understanding of $\mathrm{F=F^eF^p}$\,. {\it J. Mech. Phys. Solids} {\bf 67} (2014), 40--61.

\bibitem{RSC} {Reina, C., Schl\"omerkemper, A., Conti, S.}: Derivation of $\mathrm{F=F^{e}F^p}$  as the continuum limit of crystalline slip. {\it J. Mech. Phys. Solids} {\bf 89} (2016), 231--254.

\bibitem{Sa} {Sandier, E.}: Lower bounds for the energy of unit vector fields and applications.  {\it J. Func. Anal} {\bf 152} (1998), 379--403.

\bibitem{SS2} {Sandier, E., Serfaty, S.}: \emph{Vortices in the Magnetic Ginzburg-Landau Model}, Progress in Nonlinear Differential Equations and Their Applications, vol. 70, Birkh\"auser Boston, Boston (MA), 2007.

\bibitem{SZ} {Scardia, L., Zeppieri C.I.}: Line-tension model for plasticity as the $\Gamma$-limit of a nonlinear dislocation energy. {\it SIAM J. Math. Anal.} {\bf 44} (2012), 2372--2400.

\bibitem{Schm09} Schmidt, B.: On the derivation of linear elasticity from atomistic models. 
\textit{Netw. Heterog. Media} \textbf{4} (2009), 789--812.

\end{thebibliography}
\end{document}